\documentclass[preprint,3p,authoryear,a4paper]{elsarticle}
\usepackage{graphicx}
\usepackage{xcolor}
\usepackage{dsfont}
\usepackage{amsmath}
\usepackage{amssymb}
\usepackage{amsfonts}
\usepackage{amsbsy}
\usepackage{amsthm}
\usepackage{array}
\usepackage{booktabs} 
\usepackage{verbatim} 
\usepackage[english]{babel}
\usepackage{float}
\usepackage{subfigure}
\usepackage{epsfig}
\usepackage{pdfsync}
\usepackage{verbatim}
\usepackage{pdfpages}
\usepackage{ulem}

\newcolumntype{C}[1]{>{\centering\let\newline\\\arraybackslash\hspace{0pt}}m{#1}}
\renewcommand\appendix{\par
\setcounter{section}{0}%
\setcounter{subsection}{0}%
\setcounter{table}{0}
\setcounter{table}{0}
\setcounter{figure}{0}
\gdef\thetable{\Alph{table}}
\gdef\thefigure{\Alph{figure}}
\gdef\thesection{\Alph{section}}
\setcounter{section}{0}}

\newtheorem{theorem}{Theorem}[section]
\newtheorem{lemma}[theorem]{Lemma}
\newtheorem{proposition}[theorem]{Proposition}

\newtheorem{definition}[theorem]{Definition}

\newcommand{\corr}[1]{{#1}}
\newcommand{\corrs}[1]{\textcolor{black}{#1}}

\newcommand{\p}{^\prime}
\newcommand{\pp}{^{\prime\prime}}
\newcommand{\de}{_\delta}
\newcommand{\phiqr}{\phi_{\gamma+\delta}}
\newcommand{\Aqr}{A_{\gamma,\delta}}
\newcommand{\Auqr}{A_{u,\gamma,\delta}}
\newcommand{\Bqr}{B_{\gamma,\delta}}
\newcommand{\Kqr}{K_{\gamma,\delta}}
\newcommand{\Lqr}{L_{\gamma,\delta}}
\newcommand{\Cqr}{C_{\gamma,\delta}}
\newcommand{\Wq}{W_{\delta}}
\newcommand{\Zq}{Z_{\delta}}
\newcommand{\Zqr}{Z_{\gamma,\delta}}

\newcommand{\Jqr}{J_{\gamma,\delta}}
\newcommand{\Hqr}{H_{\gamma,\delta}}

\newcommand{\pis}{{\pi_{b_u,b_l}^{\kappa,s}}}
\newcommand{\pio}{{\pi_{b_u,b_l}^{\kappa,*}}}
\newcommand{\pioo}{{\pi_{b_u^{(1)},b_l^{(1)}}^{\kappa,*}}}
\newcommand{\piot}{{\pi_{b_u^{(2)},b_l^{(2)}}^{\kappa,*}}}
\newcommand{\Vs}{{V_s}}
\newcommand{\Vo}{{V_*}}
\newcommand{\Voo}{V_\kappa^\prime(0;\pi^{\kappa,s}_{b_u,0})}
\newcommand{\Prm}{{\mathcal{N}(ds\times dz)}}
\newcommand{\limep}{{\lim_{\epsilon\downarrow 0}}}
\newcommand{\limtn}{{\lim_{t,n\uparrow\infty}}}
\newcommand{\intte}{{\int_{0+}^{t\wedge T_n}e^{-\delta s}}}
\newcommand{\xps}{{X^\pi(s-)}}
\newcommand{\sumte}{{\sum_{0<s<t\wedge T_n}e^{-\delta s}}}

\begin{document}

\normalem 



\begin{frontmatter}

\title{Optimal periodic dividend strategies  for spectrally positive L\'evy risk processes with fixed transaction costs}

\author[UM]{Benjamin Avanzi}
\ead{b.avanzi@unimelb.edu.au}

\author[UNSW]{Hayden Lau\corref{cor}}
\ead{kawai.lau@unsw.edu.au}

\author[UNSW]{Bernard Wong}
\ead{bernard.wong@unsw.edu.au}

\cortext[cor]{Corresponding author.}

\address[UM]{Centre for Actuarial Studies, Department of Economics \\ University of Melbourne VIC 3010, Australia}
\address[UNSW]{School of Risk and Actuarial Studies, UNSW Australia Business School\\ UNSW Sydney NSW 2052, Australia}


\begin{abstract}
We consider the general class of spectrally positive L\'evy risk processes, which are appropriate for businesses with continuous expenses and lump sum gains whose timing and sizes are stochastic. Motivated by the fact that dividends \corr{cannot be} paid \corr{at any time} in real life, we study \emph{periodic} dividend strategies whereby dividend decisions are made according to a \corr{separate} arrival process. 

In this paper, we investigate the impact of fixed transaction costs on the optimal periodic dividend strategy, and show that a periodic $(b_u,b_l)$ strategy is optimal \corr{when decision times arrive according to an independent Poisson process}. Such a strategy leads to lump sum dividends that bring the surplus back to $b_l$ as long as it is no less than $b_u$ at a dividend decision time. The expected present value of dividends (net of transaction costs) is provided explicitly with the help of scale functions. Results are illustrated.


\end{abstract}

\begin{keyword}
Optimal dividends\sep Periodic dividends\sep Dual risk model\sep Fixed transaction costs\sep SPLP

JEL codes: 
C44 \sep 
C61 \sep 
G24 \sep 
G32 \sep 
G35 


\end{keyword}

\end{frontmatter}

\newtheorem{remark}{Remark}[section]
\numberwithin{equation}{section}

\section{Introduction} \label{S_intro}

The literature on the stability problem \citep*[see, e.g.][]{Buh70} is prolific. One possibility criterion for stability is the expected present value of dividends, first proposed by \citet{deF57}. A dividend strategy defines when and how much dividends should be paid, and its optimal version is \corr{(often)} the one that maximises the expected present value of dividends \citep*[see for instance][]{AlTh09}.

 Stylised models of an insurance company were first studied by pioneers such as \citet*{Lun09,Cra30,Bor67}. They considered the specific case of insurance, where income is relatively certain (premiums are determined in advance) and outflows (mainly insurance claims) are random. In this paper, we consider the class of spectrally positive L\'evy processes, whereby expenses are continuous and more certain (although still potentially perturbed by diffusion), and where the stochastic behaviour of the surplus process is on the upside, that is, gains happen at random times and with random amounts \citep*{BaKyYa12, PeYa16}. This model is sometimes referred to as ``dual model'' because it is \emph{dual} to the insurance model briefly described above \citep*[see, e.g.][]{MaRu04}. Such a model is obviously relevant to most risky business, but particularly so for commission-based businesses, pharmaceutical companies, petroleum companies \citep*{AvGeSh07}, and also for valuing venture capital investments \citep*{BaEg08}. \citet*{ChWo17} further discuss the relevance of the model, as well as its connection with queuing models. 

\corr{When dividends can be paid at any time (which we also refer to as "continuous decision making"), optimal dividend strategies for the dual model were determined in general by \citet*[without, and with fixed transaction costs, respectively]{BaKyYa12,BaKyYa13}. On the other hand, ``periodic'' dividends as introduced by \citet*{AlChTh11a} were first considered in the dual model by \citet*{AvChWoWo13}, with optimality results in \citet*{AvTuWo16,PeYa16}. When the surplus is a spectrally negative L\'evy process (with a completely monotone L\'evy density), \citet{Loe08,Loe09} showed that a barrier strategy implemented in continuous time is optimal with and without fixed transaction costs, respectively. \citet*{NoPeYaYa17} extended the result in \citet*{Loe08a} in a periodic setting (without fixed transactions costs) and showed that such class of strategy if implemented in periodic time is also optimal in the periodic setting.}

\corr{In this paper, we focus on ``periodic'' dividends, and consider \emph{fixed} transaction costs in a spectrally positive L\'evy risk process. While proportional transaction costs affect the level of the optimal barrier, they do not change results qualitatively, which makes sense as they can be interpreted as a simple change of currency. We} show that \corr{fixed transaction costs} lead to a split barrier being optimal. This again mirrors analogous results in a continuous decision making framework; see e.g. \citet*{YaYaWa11}. We further illustrate numerically that our result is consistent with \citet*{BaKyYa13} and \citet*{PeYa16} when the frequency of the dividend payment time goes to infinity and the fixed transaction costs are reduced to zero, respectively.

The paper is organised as follows. In Section \ref{section.the.model}, we define the mathematical model, whereas in Section \ref{section.candidate.strategy} we define more rigorously what the class of $(b_u,b_l)$ strategies is. We then briefly review in Sections \ref{section.definition.scale.function} and \ref{section.prelim.SPLP} scale functions and related results which  will be needed later in the paper. Section \ref{section.verification.lemma} gives a sufficient condition for a strategy to be optimal (verification lemma). Section \ref{section.value.bubl} computes the value function of a given periodic $(b_u,b_l)$ strategy. Following that, Section \ref{section.smoothness} studies the smoothness condition of the value function, which is the first step to choose our candidate strategy, where Section \ref{section.secondstep} elucidates the second step to choose our candidate strategy, which involves the derivative of the value function at the lower barrier. \corr{At the end of Section \ref{S.8} it is shown that the 2 conditions proposed in Sections \ref{section.smoothness} and \ref{section.secondstep} regarding the parameters $b_u$ and $b_l$ can always be satisfied (existence).} Section \ref{section.strategy.optimal} confirms that the candidate we constructed in Sections \ref{section.smoothness} and \ref{section.secondstep} is indeed optimal (using the verification lemma of Section \ref{section.verification.lemma}), and that it is unique. Finally, numerical illustrations are provided in Section \ref{section.illustration}. Section \ref{section.conclusion} concludes.

\section{The model}\label{section.the.model}
In this paper we use a standard set-up for stochastic processes \citep*[e.g.][Chapter 0]{Ber98}. We first define a spectrally negative L\'evy process $Y=\{Y(t);t\geq 0\}$. It is well known that the law of a L\'evy proccess can be uniquely characterised by its characteristic exponent. For $Y$, its Laplace exponent is given by 
\begin{equation}\label{def.snlp.1}
\mathbb{E}[e^{\theta Y(t)}]=e^{t\psi(\theta)}, \quad Y(0)=0,
\end{equation}
and
\begin{equation}\label{def.snlp.2}
\psi(\theta) = \psi_Y(\theta)= c\theta+\frac{\sigma^2}{2}\theta^2+\int_{(0,\infty)}(e^{-\theta s}-1+\theta s 1_{\{s<1\}})\Pi(ds),
\end{equation}
with
\begin{equation}\label{def.snlp.3}
\int_{(0,\infty)}(1\wedge z^2)\Pi(dz)<\infty,
\end{equation}
where $(c,\sigma,\Pi)$ are the L\'evy triplet of $Y$. In order to avoid trivial cases, we also require that $Y$ does not have a monotonic path. We then construct a spectrally positive L\'evy process $X$ starting at $x\in\mathbb{R}$ (initial surplus) as
\begin{align}
X(t)=x-Y(t),~t\geq 0,
\end{align}
that is, we shift the process $-Y$ upwards by $X(0):=x$ units. We denote its law by $\mathbb{P}_x$ and the mathematical expectation operator related to it as $\mathbb{E}_x[\cdot]$. Next, we define periodic dividend decision times (the time where one has to decide how much to pay and the payment occurs \corrs{instantaneously}), or in short decision times. Decision times are the times when the Poisson process \corr{(independent of $X$)} with rate $\gamma$, $N_\gamma(t)$, has increments, i.e. the set $\mathbb{T}:=\{T_i;i\in\mathbb{N}\}$, with
\begin{equation}
T_i=\inf\{t:N_\gamma(t)=i\},
\end{equation}
\corr{where throughout this paper we adopt the convention 
	\begin{equation}\label{eq.convention.inf.empty}
	\inf\emptyset=\infty.
	\end{equation} }
In words, it means that the $i$-th decision time, $T_i$, corresponds to the time when $N_\gamma$ jumps from $i-1$ to $i$. Let $\mathbb{F}:=\{\mathcal{F}(t);t\geq 0\}$ be the filtration generated by the process $(X,N_\gamma)$. Then, a periodic dividend strategy (defined by the cumulative dividends paid) $\pi:=D^\pi=\{D^\pi(t);t\geq 0\}$ is a non-decreasing, right-continuous and $\mathbb{F}$-adapted process, which admits the form
\begin{equation}\label{periodic.dividend.form}
 D^\pi(t) = \int_{[0,t]}\nu^\pi(s)dN_\gamma(s),~t\geq 0,~\corr{\text{with}~D^\pi(0)=0.}
\end{equation}

Hence, the dividend amount paid at $T_i$ is $\xi^\pi_i:=\nu^\pi(T_i)$ (the increment of $D^\pi$ at $T_i$) and the strategy $\pi$ can also be specified in terms of $\{\xi^\pi_i;i\in\mathbb{N}\}$. The modified surplus $X^\pi=\{X^\pi(t);t\geq 0\}$ is defined as
\begin{align}
X^\pi(t)=X(t)-D^\pi(t),
\end{align}
where the ruin time $\tau^\pi$ is defined as
\begin{align}
	\tau^\pi=\inf\{t\geq 0:X^\pi(t)< 0\},
\end{align}
the instant that the modified surplus goes below $0$ for the first time.

We now introduce the constraints for the periodic dividend strategy. Intuitively, given that a fixed transaction cost $\kappa>0$ is incurred on each dividend payment, the (gross) amount of dividend should be large enough to pay the transaction cost, i.e.
\begin{equation}\label{Pikappa}
\xi_i^\pi \geq\kappa~\mbox{if}~\xi^\pi_i\neq 0.
\end{equation}
This holds naturally (see property 5 in Remark \ref{remark.different.kappa.on.v} below). Since we are not allowed to inject capital into the surplus, and since a dividend payment cannot exceed the current surplus, we have the following restrictions:
\begin{equation}
X^\pi(T_i)\geq 0,~T_i\leq \tau^\pi,
\end{equation}
which translates to 
\begin{equation}
0\leq \xi^\pi_i\leq X^\pi(T_i-)~\forall~i\in\mathbb{N},
\end{equation}
as the jump times of $X$ and the periodic decision times are distinct with probability $1$.

We can see from the above definition that, at a decision time, a decision to not pay any dividend is also allowed. In this case,  no transaction cost is incurred. It is also possible that a dividend payment can cause ruin, which refers to liquidation of the company, i.e. the company chose to close its business by distributing all the available surplus (at its first opportunity). This strategy is called a \emph{liquidation-at-first-opportunity} strategy. We denote the set of all admissible strategies $\Pi$ and define $\Pi_\kappa$ the set of all admissible strategies such that (\ref{Pikappa}) holds.

Lastly, we introduce the time preference parameter $\delta>0$. The value function of a strategy $\pi,~\pi\in\Pi$ with initial surplus $x$ is denoted as, $V_{\kappa}(x;\pi)$. We define
\begin{align}
V_{\kappa}(x;\pi):=\mathbb{E}_x\left[\sum_{i=1}^{\infty}e^{-\delta T_i}(\corr{\xi_i^\pi}-\kappa)1_{\{\xi_i>0\}}1_{\{T_i\leq\tau^\pi\}}\right].\label{def.value.strategy}
\end{align}

Note that we have $X^\pi(\tau^\pi)=0$ and $V_{\kappa}(0;\pi)=0$ for all strategies $\pi\in\Pi$ since ruin at $t=0+$ with $X(0)=0$ is certain for a spectrally positive L\'evy process \corr{with no monotonic paths}. Our goal is to find an optimal strategy $\pi^*_{\kappa}$ (if it exists) such that
\begin{align}\label{eqt.optimal}
&V_{\kappa}(x;\pi_{\kappa}^*)=v_{\kappa}(x):=\sup_{\pi\in\Pi} V_{\kappa}(x;\pi).
\end{align}

\begin{remark}\label{remark.different.kappa.on.v}
	From the definitions of $V_{\kappa}$ and $\Pi_\kappa$, we have for any $0\leq\kappa_1\leq \kappa_2$
	\begin{enumerate}
		\item $\pi\in\Pi_{\kappa_2}\implies\pi\in\Pi_{\kappa_1}$ and
		\item $\pi\in\Pi_{\kappa_2}\implies V_{\kappa_1}(x;\pi)\geq V_{\kappa_2}(x;\pi)$ for all $x\geq 0$, and
		\item $v_{\kappa_1}(x)\geq v_{\kappa_2}(x)$ for all $x\geq 0$, and
		\item $V_{\kappa}(x;\pi)\geq 0$ for all $x\ge 0$ and $\pi \in\Pi_\kappa $, and 
		\item $v_{\kappa}(x)=\sup_{\pi\in\Pi_\kappa} V_{\kappa}(x;\pi)$.
	\end{enumerate}
\end{remark}
\begin{proof}[Proof of 5]
	Note this property justifies our statement just after \eqref{Pikappa}. To prove this property, it suffices to show that all strategies in $\Pi\backslash\Pi_\kappa$ can be outperformed by the strategies in $\Pi_\kappa$. If a strategy $\pi$ is in $\Pi\backslash\Pi_\kappa$, there are some dividend \corr{payments} $\xi_i$ smaller than $\kappa$. Those will contribute a negative value to the value function. By choosing not to pay dividends at those dividend decision \corr{times}, call it strategy $\pi_\kappa$, we can remove those negative \corr{contributions} while having a higher surplus level at those times, resulting in a smaller probability of ruin, or $X^\pi(t)\leq \corr{X^{\pi_{\kappa}}(t)}$ for all $t\leq \tau^\pi$ and $\xi^\pi_i-\kappa\leq \xi^{\pi_{\kappa}}_i-\kappa$ for all $i\in\mathbb{N}$. Thus, we have $V_{\kappa}(x;\pi)\leq V_{\kappa}(x;\pi_\kappa)$. Finally we note that $\pi_\kappa$ either pays dividends above or equal to $\kappa$, or does not pay any dividend, thus is inside the set $\Pi_\kappa$.
\end{proof}

Thanks to the fifth property in Remark \ref{remark.different.kappa.on.v}, it suffices to consider only the strategies in $\Pi_\kappa$. Therefore, in the remaining of this paper, we restrict ourselves to strategies in $\Pi_{\kappa}$.

\section{Candidate strategy}\label{section.candidate.strategy}
Inspired by the form of optimal strategies in the literature, e.g. \citet*{BaKyYa13}, \citet*{PeYa16}, \citet*{Loe08a} and \citet*{NoPeYaYa17}, we conjecture that an optimal  periodic strategy will be of the form $(b_u,b_l)$, as defined in Definition \ref{D_bubl} and illustrated in Figure \ref{fig.bubl}.

\begin{definition}[Periodic ($b_u,b_l$) strategy]\label{D_bubl}
	A periodic $(b_u,b_l) $ strategy with $0\leq b_l\leq b_u$ is the strategy that pays $x-b_l$ whenever the surplus $x$ is above or equal to $b_u$, at decision times. This reduces the surplus level to $b_l$.
	
	\begin{figure}[H]
		\centering
		\includegraphics[width=0.7\textwidth]{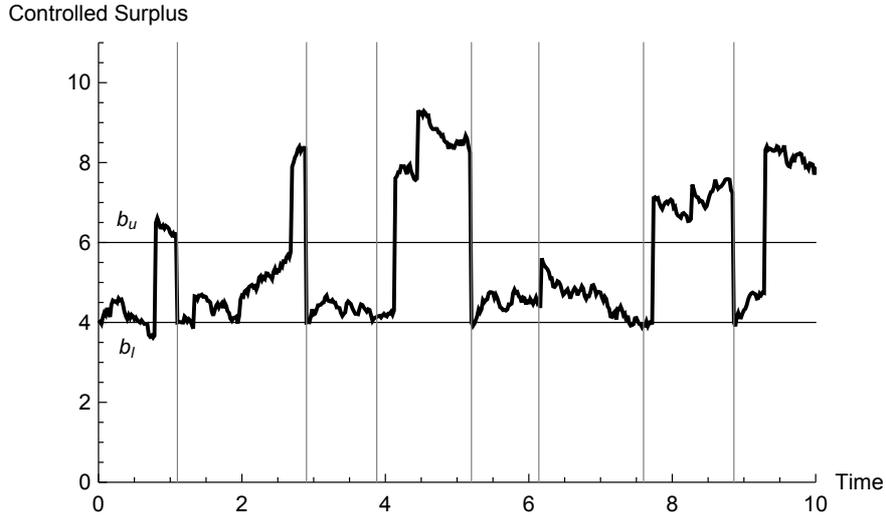}
		\caption{An illustration of a periodic $(b_u,b_l)$ strategy. The vertical lines represent the (Poissonian) dividend decision times.}
		\label{fig.bubl}
	\end{figure}
\end{definition}

By denoting the strategy as $\pi_{b_u,b_l}$, we have
\begin{align}\label{divpis}
	\xi_i^{\pi_{b_u,b_l}}=[X^{\pi_{b_u,b_l}}(T_i-)-b_l]1_{\{X^{\pi_{b_u,b_l}}(T_i-)\geq b_u\}}.
\end{align}
Clearly, we have 
\begin{equation*}
\pi_{b_u,b_l}\in\Pi_\kappa\iff b_u-b_l\geq\kappa.
\end{equation*}

\begin{remark}
	Similarly, a periodic barrier strategy at barrier level $b>0$, denoted as $\pi_{b}$, is defined as
	\begin{align}\label{divpbs}
		\xi_i^{\pi_b}=[X^{\pi_{b}}(T_i-)-b]1_{\{X^{\pi_b}(T_i-)\geq b\}}.
	\end{align}
\end{remark}

\section{Definition of scale functions}\label{section.definition.scale.function}

This section gives definitions of scale functions for our purpose, that is, to calculate the value function of a periodic $(b_u,b_l)$ strategy. The central idea of scale function is based on path decomposition of L\'evy processes, i.e. Wiener-Hopf factorisation and It\^o's excursion theory. Interested readers can refer to good textbooks including \citet*{Ber98} and \citet*{Kyp06}. Additional references regarding the general theories of spectrally negative L\'evy processes includes \citet*{ChDo05}, \citet*{LoReZh14}, \citet*{PaPeRi15} and \citet*{PeYa16b}, where some useful identities are also available. Examples of the applications of fluctuation theories include \citet*{Fur98}, \citet*{AvPaPi07}, \citet*{ChYaZh17P} and \citet*{YaYaZh13}. Recently, \citet*{AvGrVa17} summarised the identities and the applications of fluctuation theories on risk theory. 

The $q$-scale function, $W_q$, for $x\geq0$, $q\geq 0$ is defined through the inverse Laplace transform of $1/(\psi(\theta)-q)$, i.e.
\begin{equation*}
\int_{0}^{\infty}e^{-\theta x}W_q(x)dx=\frac{1}{\psi(\theta)-q},~~\theta>\phi_q,
\end{equation*}
where 
\begin{equation*}
\phi_q = \sup \{s:\psi(s)=q\}.
\end{equation*}
Moreover, for $x\geq0$ the ``tilted'' $q$-scale function is defined as
\begin{equation}
Z_q(x,\theta)=e^{\theta x}\left(1-(\psi(\theta)-q)\int_{0}^{x}e^{-\theta y}W_q(y)dy\right).
\end{equation}
In particular, we also define for $r>0$
\begin{align}
Z_{r,q}(x):=~&Z_q(x,\phi_{r+q})\corr{~=r\int_0^\infty e^{-\phi_{r+q}u}W_{q}(x+u)du},\label{Zqr.2}\\
Z_q(x):=~&Z_q(x,0)=1+q\int_{0}^{x}W_q(y)dy,
\end{align} 
\corr{where the equality in \eqref{Zqr.2} is due to $\int_0^\infty e^{-\phi_{r+q} u} W_q(u)du=1/r$.}

For $x<0$, we define 
\begin{align}
W_q(x)&=0,\label{scale.fcn.neg.beg}\\
Z_q(x)&=1,\\
Z_{r,q}(x)&=e^{\phi_{r+q} x},\\
{\overline{Z}}_q(x)&=x.\label{scale.fcn.neg.end}
\end{align}
We also define the integral of functions by adding an overhead line to them, such that 
\begin{align*}
{\overline{W}}_q(x)&=\int_{0}^{x}W_q(y)dy,\\
{\overline{\overline{W}}}_q(x)&=\int_{0}^{x}{\overline{W}}_q(y)dy,\\
{\overline{Z}}_q(x)&=\int_{0}^{x}Z_q(y)dy,\nonumber
\end{align*}
\corr{where we also define $W_q(x)=0$ for $x<0$.}

In addition, we also define $\mu :=-\psi'(0+)<\infty$ and 
\begin{align}
J_{r,q}(x) :=& \frac{q}{r+q}Z_{r,q}(x)+\frac{r}{r+q}Z_q(x),\\
H_{r,q}(x) :=& \frac{r}{r+q}\left( \bar{Z}_q(x)-\frac{\mu}{q}\right).
\end{align}


\section{Preliminaries: \corr{results on periodic barrier strategies}}\label{section.prelim.SPLP}

In this section, we list some useful results \corr{related to periodic barrier strategies} which will be used in later sections. First, we present the following identity extracted from Equation (5) in \citet*{AlIvZh16}: for $0\leq x\leq y$,
\begin{equation}
\mathbb{E}(e^{-\delta \tau^-_{Y,0,x} +\phiqr Y(\tau^-_{Y,0,x})};\tau^-_{Y,0,x}<\tau^+_{Y,y,x})=\Zqr(x)-W_\delta(x)\frac{\Zqr(y)}{W_\delta(y)},\label{ineq.Zqr.W}
\end{equation}
where $\tau^-_{Y,a,x} = \inf\{t\geq0:x+Y(t)<a\}$ and $\tau^+_{Y,a,x}=\inf\{t\geq0:x+Y(t)>a\}$ for any $a\in\mathbb{R}$. Hence, we have 
\begin{equation}\label{ineqt.Z.W}
\frac{\Zqr(x)}{\Wq(x)}\geq \frac{\Zqr(y)}{\Wq(y)},~0\leq x\leq y,
\end{equation}
\corr{which} is analogous to $\Wq(x)/\Wq^\prime(x)\geq \Wq(y)/\Wq^\prime(y)$ in \citet*{BaKyYa13}.

\corr{
In addition, using existing results in scale function as in Appendix \ref{Appendix.Formula.ZqrZq}, one can deduce that
\begin{equation}
\frac{\partial}{\partial x}\frac{\Zqr(x)}{\Zq(x)}> 0.\label{eqt.ZqrZq.increasing}
\end{equation}
}

From \citet*{PeYa16}, we know that
\begin{equation}
V_{0}(x;\pi_b)=\frac{\Hqr(b)}{\Jqr(b)}\Jqr(b-x)-\Hqr(b-x),~x\geq 0,~b>0,
\end{equation}
and that
\begin{equation}
v_{0}(x)=\frac{\Hqr(b^*)}{\Jqr(b^*)}\Jqr(b^*-x)-\Hqr(b^*-x),~x\geq 0, \label{eqt.v0rho}
\end{equation}
where $b^*$ is a periodic barrier, and is the unique solution of 
\begin{equation}\label{eqt.Q}
Q(b):=\frac{\Hqr(b)}{\Jqr(b)}+\frac{1}{\phiqr}=0,
\end{equation}
given that $Q(0)<0$. Otherwise, $b^*=0$. In addition, if $b^*>0$, $Q(x)<0$ (resp. $Q(x)>0$) if $x<b^*$ (resp. $x>b^*$). 
\corrs{Hence, we can deduce that for $b^*>0$
	\begin{equation}\label{V0bstar}
	V_0(b^*;\pi_{b^*})= \frac{\gamma}{\gamma+\delta}\frac{\mu}{\delta}-\frac{1}{\phiqr}.
	\end{equation}
}
	Furthermore, we have 
\begin{equation}\label{Eqt.VdEq1bstar}
V_0^\prime(b^*;\pi_{b^*})\leq 1.
\end{equation}

To supplement the results in \citet*{PeYa16}, we establish the following \corr{results}:
\begin{lemma}\label{lem.vprime.geq0}It holds that
\begin{equation}\label{eqt.vprime.geq0}
V_{0}^\prime(x;\pi_b)\geq 0,~x\geq 0.
\end{equation}
\end{lemma}

\begin{proof}[Proof of Lemma \ref{lem.vprime.geq0}] 
 It suffices to show that $V_0(x;\pi_{b})$ is strictly increasing in $x$, i.e.
\begin{equation*}
V_{0}(x;\pi_{b})<V_{0}(y;\pi_{b}),~\mbox{for }0\leq x<y.
\end{equation*}

\corrs{Regarding the construction of the law $\mathbb{P}$ of $X$ on the probability space of the sample path, $(\Omega, \mathbb{F})$, we shall assume that $\mathbb{P}(X(0)=0)=1$ and refer to the law $\mathbb{P}_x$ the law of $x+X$ under $\mathbb{P}$. In this sense, in the following, we will only work with the primary measure $\mathbb{P}$, which have also taken into account of the $\{T_i\}$ (for example by taking the product measure).}

Suppose $x>0$, for a given sample path of $X=X(\omega)$ starting at $0$, the modified sample path with $x$ units shifted upward (denoted $X^{\pi_{b},x}_\omega$) can never exceed that with $y$ units shifted up (denoted $X^{\pi_{b},y}_\omega$), because a dividend either brings both surplus down to $b$ ($X^{\pi_{b},y}_\omega(T_i)=b=X^{\pi_{b},x}_\omega(T_i)$ for some $i\in\mathbb{N}$), or brings the $X^{\pi_b,y}_\omega$ closer to $X^{\pi_b,x}_\omega$ such that $X^{\pi_{b},y}_\omega(T_i)=b>X^{\pi_{b},x}_\omega(T_i)$ for some $i\in\mathbb{N}$. Using the same argument, it is clear that $X^{\pi_{b},x}_\omega$ can never receive more \corr{dividends} than $X^{\pi_{b},y}_\omega$. This shows $V_{0}(x;\pi_{b})\leq V_{0}(y;\pi_{b})$. Note on the ($\mathbb{P}$-a.s.) event $\{T_1<\infty\}$, there is a positive probability that $x+X(\omega)$ hits above $b_u$ at $T_1$ before ruin, giving the strict inequality.

For $x=0$, we have $V_{0}(x;\pi_{b})=0$. Hence, it suffices to show $V_0(y;\pi_{b})>0$. Again, on the ($\mathbb{P}$-a.s.) event $\{T_1<\infty\}$, there is a positive probability that $y+X(\omega)$ hits above $b_u$ at $T_1$ before ruin, showing the strict inequality.
\end{proof}
\corr{
\begin{lemma}\label{Lemma.VdLeq1AtBstar}
	For $d> 0$, we have: 
	\begin{equation}\label{Vp.bublSmoothLeq1}
	V_0^\prime(b^*;\pi_{b^*+d})=-\frac{\Hqr(b^*+d)}{\Jqr(b^*+d)}\Jqr^\prime(d)+\Hqr^\prime(d)<1.
	\end{equation}
\end{lemma}
\begin{proof}
	By direct computation, we obtain for $b\geq 0$
	\begin{align}
	\frac{\partial}{\partial b}\Big(\frac{\Hqr(b)}{\Jqr(b)}\Big)=~&1-\Big(\frac{\Hqr(b)}{\Jqr(b)}+\frac{1}{\phiqr}\Big)\frac{\Jqr^\prime(b)}{\Jqr(b)},\\
	\frac{1}{\phiqr}\Jqr^{\pp}(b)=~&\Jqr^\prime(b)-\Hqr^{\pp}(b),
	\end{align}
	which gives
	\begin{align*}
	~&\frac{\partial}{\partial d}\Big(-\frac{\Hqr(b^*+d)}{\Jqr(b^*+d)}\Jqr^\prime(d)+\Hqr^\prime(d)\Big)\\
	=~&-\Jqr^\prime(d)\Bigg(1-\Big(\frac{\Hqr(b^*+d)}{\Jqr(b^*+d)}+\frac{1}{\phiqr}\Big)\frac{\Jqr^\prime(b^*+d)}{\Jqr(b^*+d)}\Bigg)-\frac{\Hqr(b^*+d)}{\Jqr(b^*+d)}\Jqr^{\pp}(d)+\Hqr^{\pp}(d)\\
	=~&-\Big(\frac{\Hqr(b^*+d)}{\Jqr(b^*+d)}+\frac{1}{\phiqr}\Big)\frac{1}{\Jqr(b^*+d)}\Big(\Jqr^{\pp}(d)\Jqr(b^*+d)-\Jqr^\prime(b^*+d)\Jqr^\prime(d)\Big)\\
	=~&-\Big(\frac{\Hqr(b^*+d)}{\Jqr(b^*+d)}+\frac{1}{\phiqr}\Big)\frac{1}{\Jqr(b^*+d)}\Big(\phiqr\Jqr^\prime(d)\Jqr(b^*+d)-\phiqr\Hqr^{\pp}(d)\Jqr(b^*+d)-\Jqr^\prime(b^*+d)\Jqr^\prime(d)\Big)\\
	=~&-\Big(\frac{\Hqr(b^*+d)}{\Jqr(b^*+d)}+\frac{1}{\phiqr}\Big)\frac{\phiqr\frac{\gamma}{\gamma+\delta}}{\Jqr(b^*+d)}\Big(\frac{\delta}{\gamma+\delta}\phiqr\Zqr(d)\Zq(b^*+d)-\delta\Wq(d)\Jqr(b^*+d)\Big)
	\end{align*}
	where the last term inside the bracket is
	\begin{align*}
		&\frac{\delta}{\gamma+\delta}\phiqr\Zqr(d)\Zq(b^*+d)-\delta\Wq(d)\frac{\delta}{\gamma+\delta}\Zqr(b^*+d)-\delta\Wq(d)\frac{\gamma}{\gamma+\delta}\Zq(b^*+d)\\
		=~&\frac{\delta}{\gamma+\delta}\Big(\Zq(b^*+d)\big(\phiqr\Zqr(d)-\gamma\Wq(d)\big)-\delta\Wq(d)\Zqr(b^*+d)\Big)\\
		=~&\frac{\delta}{\gamma+\delta}\Big(\Zq(b^*+d)\Zqr^\prime(d)-\Zqr(b^*+d)\Zq^\prime(d)\Big)
	\end{align*}
	By taking derivative w.r.t. $b$, we get
	\begin{align*}
	\frac{\partial}{\partial b}\Big(\Zq(b+d)\Zqr^\prime(d)-\Zqr(b+d)\Zq^\prime(d)\Big)
	=~&\delta \Wq(b+d)\Zqr^\prime(d)-\Zqr^\prime(b+d)\delta \corrs{\Wq(d)}\\
	=~&\phiqr\delta\Wq(b+d)\Big(\Zqr(d)-\frac{\Wq(d)}{\Wq(b+d)}\Zqr(b+d)\Big)\\
	\geq ~& 0
	\end{align*}
	by \eqref{ineq.Zqr.W}. Therefore, we can conclude that 
	\begin{align*}
	~&\frac{\partial}{\partial d}\Big(-\frac{\Hqr(b^*+d)}{\Jqr(b^*+d)}\Jqr^\prime(d)+\Hqr^\prime(d)\Big)\\
	=~&-\Big(\frac{\Hqr(b^*+d)}{\Jqr(b^*+d)}+\frac{1}{\phiqr}\Big)\frac{\phiqr\frac{\gamma}{\gamma+\delta}}{\Jqr(b^*+d)}\Big(\frac{\delta}{\gamma+\delta}\phiqr\Zqr(d)\Zq(b^*+d)-\delta\Wq(d)\Jqr(b^*+d)\Big)\\
	\leq ~&-\Big(\frac{\Hqr(b^*+d)}{\Jqr(b^*+d)}+\frac{1}{\phiqr}\Big)\frac{\phiqr\frac{\gamma}{\gamma+\delta}}{\Jqr(b^*+d)}\Big(\frac{\delta}{\gamma+\delta}\phiqr\Zqr(d)\Zq(d)-\delta\Wq(d)\Jqr(d)\Big)\\
	=~&-\Big(\frac{\Hqr(b^*+d)}{\Jqr(b^*+d)}+\frac{1}{\phiqr}\Big)\frac{\phiqr\frac{\gamma}{\gamma+\delta}}{\Jqr(b^*+d)}\frac{\delta}{\gamma+\delta}\Big(\Zq(d)\Zqr^\prime(d)-\Zq^\prime(d)\Zqr(d)\Big)\\
	=~&-\Big(\frac{\Hqr(b^*+d)}{\Jqr(b^*+d)}+\frac{1}{\phiqr}\Big)\frac{\phiqr\frac{\gamma}{\gamma+\delta}}{\Jqr(b^*+d)}\frac{\delta}{\gamma+\delta}\Zq(d)^2\frac{\partial}{\partial d}\Big(\frac{\Zqr(d)}{\Zq(d)}\Big)\\
	<~& 0
	\end{align*}
	by \eqref{eqt.ZqrZq.increasing}. Finally, by integrating the above and using \eqref{Eqt.VdEq1bstar} we get
	$$V_0^\prime(b^*;\pi_{b^*+d})=-\frac{\Hqr(b^*+d)}{\Jqr(b^*+d)}\Jqr^\prime(d)+\Hqr^\prime(d)<V_0^\prime(b^*;\pi_{b^*})\leq 1$$as required.
\end{proof}
}

\section{Verification lemma}\label{section.verification.lemma}
In this section, we give a sufficient condition for a strategy to be optimal. We first characterise the smoothness of a function by the following definition.
\begin{definition}\label{def.smooth.function}
	\corr{If $X$ is of unbounded variation, a function $f$ is smooth if $f\in\mathcal{C}^2(0,\infty)$. Otherwise if $X$ is of bounded variation, a function $f$ is smooth if $f\in\mathcal{C}^1(0,\infty)$.}
\end{definition}
\begin{remark}\label{remark.X.variations}
	It is well-known that $X$ has path of bounded variation if and only if $\sigma=0$ and $\int_{(0,1)}x\Pi(dx)<\infty$. Examples of SPLP with bounded variation include the compound Poisson process with negative drift. Examples of SPLP with unbounded variation include \corr{Brownian motion} with upward jumps, see Section \ref{section.illustration} for illustrations.
\end{remark}

The extended generator for $X$ applied on a function $F$, $\mathcal{L}F$, is given by
\begin{equation}\label{Def.extended.operation}
\mathcal{L}F(x):=-cF'(x)+\frac{\sigma^2}{2}F''(x)+\int_{(0,\infty)}\big[F(x+s)-F(x)-F'(x)s1_{\{s<1\}}\big]\Pi(ds)
\end{equation}
whenever it ($\mathcal{L}F$) is well defined (if $F$ is sufficiently smooth), and where the term $\frac{\sigma^2}{2}F''(x)$ is understood to vanish if $X$ is of bounded variation (no Gaussian component). Lemma \ref{verification.lemma} characterises sufficient conditions that a strategy need to satisfy in order to be optimal.

\begin{lemma}\label{verification.lemma}
	Suppose $\pi\in\Pi_\kappa$ and its value function $H(x):=V_{\kappa}(x;\pi)$ satisfies
	\begin{enumerate}
		\item $H$ is smooth, 
		\item $H\geq 0$,
		\item $(\mathcal{L}-\delta)H(x)+\gamma\left( l-\kappa+H(x-l)-H(x)\right)\leq 0$,  $x\geq l \geq 0$,
	\end{enumerate}
	then $\pi$ is optimal, i.e. $H(x)=v_{\kappa}(x)$ for all $x\geq0$.
\end{lemma}
\begin{proof}[Proof of Lemma \ref{verification.lemma}]
	See Appendix \ref{Appendix.A}.
\end{proof}

\section{Value function of a periodic $(b_u,b_l)$ strategy}\label{section.value.bubl}
In the following, we first calculate the value function of a periodic $(b_u,b_l)$ strategy with any choices of $b_u$ and $b_l$ such that $b_u>b_l\geq 0$. The value function of $\pi_{b_u,b_l}$ is given by the following theorem. 
\begin{theorem}\label{lemma.value.fcn}
	For $x\geq0$, the value function is given by
	\begin{equation}
	V_{\kappa}(x;\pi_{b_u,b_l})=\frac{\gamma}{\gamma+\delta}\Aqr(b_u-x;d)+\Zqr(b_u-x) V_{\kappa}(b_u;\pi_{b_u,b_l}) + \Cqr(b_u-x) V_{\kappa}(b_l;\pi_{b_u,b_l})\label{V.kapparho},
	\end{equation}
	where
		\begin{align}
	\frac{\gamma}{\gamma+\delta}\Aqr(x;d) &= -\frac{\gamma}{\gamma+\delta}\frac{\mu}{\delta}\Jqr(x) -\Hqr(x)+( d-\kappa )\Cqr(x),\label{A.expression} \\
	d&=b_u-b_l,\\
	\Cqr(x)&=\frac{\gamma}{\gamma+\delta}(\Zq(x)-\Zqr(x)).
	\end{align}
	The constants $V_{\kappa,}(b_u;\pi_{b_u,b_l})$ and $ V_{\kappa}(b_l;\pi_{b_u,b_l})$ are given by 
	\begin{align}
	V_\kappa(b_l;\pi_{b_u,b_l})=~&\frac{(d-\kappa-\frac{\gamma}{\gamma+\delta}\frac{\mu}{\delta})\Zqr(b_u)\Lqr(d;b_u)+\Hqr(b_u)\Zqr(d)-\Hqr(d)\Zqr(b_u)}{\Zqr(b_u)(1-\Lqr(d;b_u))},\label{value.bl}\\
	V_\kappa(b_u;\pi_{b_u,b_l})=~&\frac{\frac{-\gamma}{\gamma+\delta}\Aqr(b_u;d)-\frac{\gamma}{\gamma+\delta}\frac{\mu}{\delta}\Zqr(b_u)\Lqr(d;b_u)+\Cqr(b_u)\Hqr(d)-\Cqr(d)\Hqr(b_u)}{\Zqr(b_u)(1-\Lqr(d;b_u))}\label{value.bu}
	\end{align}
	if $b_l\neq 0$
	\begin{equation}
	\Lqr(x;b_u) :=  \frac{\gamma}{\gamma+\delta}\left[ \Zq(x)-\frac{\Zqr(x)}{\Zqr(b_u)}\Zq(b_u)\right].
	\end{equation}
 Otherwise, we have
	\begin{align}
	V_\kappa(b_l;\pi_{b_u,b_l})&=0,\\
	V_\kappa(b_u;\pi_{b_u,b_l})&=\frac{\gamma}{\gamma+\delta}\frac{-\Aqr(b_u;d)}{\Zqr(b_u)}.
	\end{align}
\end{theorem}
\begin{proof}[Proof of Theorem \ref{lemma.value.fcn}]
 See Appendix \ref{Appendix.B}. 
\end{proof}

\begin{remark}
	With the help from (\ref{scale.fcn.neg.beg})-(\ref{scale.fcn.neg.end}), we can rewrite (\ref{V.kapparho}) for $x\geq b_u$ as
	\begin{equation}
	V_{\kappa}(x;\pi_{b_u,b_l})=\frac{\gamma}{\gamma+\delta} \Auqr(x-b_u;d)+e^{-\phiqr(x-b_u)} V_{\kappa}(b_u;\pi_{b_u,b_l}) + \frac{\gamma}{\gamma+\delta} \left(1-e^{-\phiqr(x-b_u)}\right) V_{\kappa}(b_l;\pi_{b_u,b_l})\label{V.kapparho.upper},
	\end{equation}
	where 
	\begin{align}
	\Auqr(x;d)= x +\left(d-\kappa+\frac{\mu}{\gamma+\delta}\right)(1-e^{-\phiqr x}).
	\end{align}
	which is a more explicit representation.
\end{remark}

\section{Construction of a candidate optimal strategy}\label{S.8}

In this section, we construct a candidate optimal strategy, making two educated guesses for the optimality conditions, which we implement sequentially: smoothness (Section \ref{section.smoothness}), and a further condition on the derivative at the optimal lower barrier (Section \ref{section.secondstep}). Existence is established \corr{at the end of Section \ref{section.secondstep}}, but proof \corr{of} uniqueness is postponed until the end of Section \ref{section.strategy.optimal}, where we verify that it is indeed the optimal strategy thanks to the verification lemma developed in Section \ref{section.verification.lemma}.

\subsection{First step: a smoothness condition}\label{section.smoothness}

Based on a smooth fitting argument, the optimal value function should have one more degree of smoothness (compared to that of a general $(b_u,b_l)$ strategy). In this section, we investigate which condition the value function $V_{\kappa}(\cdot;\pi_{b_u,b_l})$ must satisfy in order to be  smooth according to Definition \ref{def.smooth.function}. This is summarised in the following lemma.

\begin{lemma}\label{lemma.smooth}
	$V_{\kappa}(\cdot;\pi_{b_u,b_l})$ is smooth if and only if 
	\begin{equation}
	V_{\kappa}(b_u;\pi_{b_u,b_l})-V_{\kappa}(b_l;\pi_{b_u,b_l})=b_u-b_l-\kappa.\label{smoothness.condition}
	\end{equation}
\end{lemma}

\begin{proof}[Proof of Lemma \ref{lemma.smooth}]
	First, note that
	\begin{align*}
	&\Zqr\p(x)=\phiqr \Zqr(x)-\gamma \Wq(x),\\
	&\Zqr\pp(x)= \phiqr^2 \Zqr(x) - \phiqr\gamma \Wq(x) - \gamma \Wq\p(x);
	\end{align*}
	and also that
	\begin{equation*}
	\Zq(0)=~1,\quad
	\Zq\p(x)=~\delta \Wq(x),\quad \text{and}\quad
	\Zq\pp(x+) =~ \delta \Wq\p(x+),
	\end{equation*}
	$x\geq 0$. In addition, it is well known that 
	\begin{equation*}
	W_q(0)=0 \iff X~\mbox{is of unbounded variation}.
	\end{equation*}	
	From $\Auqr(x;d)=x+(b_u-b_l-\kappa+\frac{\mu}{\gamma+\delta})(1-e^{-\phiqr x})$, we have
	\begin{equation*}
	\Auqr'(0+;d)=1+\phiqr (d-\kappa+\frac{\mu}{\gamma+\delta}), \text{ and }\Auqr''(0+;d)=-\phiqr^2(d-\kappa+\frac{\mu}{\gamma+\delta}).
	\end{equation*}
	Therefore, we get from \corr{Theorem} \ref{lemma.value.fcn} that
	\begin{align*}
	&V_{\kappa}'(b_u+;\pi_{b_u,b_l}) = \frac{\gamma}{\gamma+\delta}\left(1+\phiqr (d-\kappa+\frac{\mu}{\gamma+\delta})\right) -\phiqr V_{\kappa}(b_u;\pi_{b_u,b_l}) + \frac{\gamma}{\gamma+\delta}\phiqr V_{\kappa}(b_l;\pi_{b_u,b_l}),\\
	&V_{\kappa}''(b_u+;\pi_{b_u,b_l}) =  \frac{\gamma}{\gamma+\delta}\left( -\phiqr^2(d-\kappa+\frac{\mu}{\gamma+\delta})\right) +\phiqr^2 V_{\kappa}(b_u;\pi_{b_u,b_l}) -\frac{\gamma}{\gamma+\delta}\phiqr^2 V_{\kappa}(b_l;\pi_{b_u,b_l}).
	\end{align*}
	On the other hand, we have
	\begin{align*}
	\Aqr'(0+;d)=&-\frac{\mu}{\delta}\left(\frac{\delta}{\gamma+\delta}\left( \phiqr -\gamma \Wq(0+)\right) +\frac{\gamma}{\gamma+\delta}\delta \Wq(0+) \right) -1+\left( d-\kappa \right) \left( \delta \Wq(0+) - \phiqr +\gamma \Wq(0+)\right)\nonumber\\
	=&-\frac{\mu}{\gamma+\delta}\phiqr-1+\left( d-\kappa \right) \left( (\gamma+\delta) \Wq(0) - \phiqr \right)\nonumber\\
	=&-1-\phiqr\left( d-\kappa+\frac{\mu}{\gamma+\delta} \right)
	+\left( d-\kappa \right)(\gamma+\delta) \Wq(0)
	\end{align*}
	when $X$ is of bounded variation and 
	\begin{align*}
	\Aqr''(0+;d)=&-\frac{\mu}{\delta}\left(\frac{\delta}{\gamma+\delta}
	\left(  \phiqr^2  - \phiqr\gamma \Wq(0+) - \gamma \Wq^{\prime}(0+)\right) +\frac{\gamma}{\gamma+\delta}\delta \Wq^{\prime}(0+) \right)-\delta \Wq(0)\nonumber\\ 
	&+\left( d-\kappa \right) \left( \delta \Wq^{\prime}(0+) - \phiqr^2  + \phiqr\gamma \Wq(0) + \gamma \Wq^{\prime}(0+)\right)\nonumber\\
	=&-\frac{\mu}{\gamma+\delta}
	\left(  \phiqr^2  - \phiqr\gamma \Wq(0) \right) -\delta \Wq(0)\nonumber\\ 
	&+\left( d-\kappa \right) \left( (\gamma+\delta) W^{\prime}(0+) - \phiqr^2  + \phiqr\gamma \Wq(0) \right)\\
	=&-\frac{\mu}{\gamma+\delta}\phiqr^2 +\left( d-\kappa \right) \left( (\gamma+\delta) \Wq^{\prime}(0+) - \phiqr^2 \right)\nonumber\\
	=&\left( d-\kappa \right)  (\gamma+\delta) \Wq^{\prime}(0+)-\phiqr^2\left(d-\kappa+\frac{\mu}{\gamma+\delta}\right)
	\end{align*}
	when $X$ is of unbounded variation. Similarly,
	\begin{align*}
	\Zqr\p(0+) &= \phiqr -\gamma \Wq(0),\\
	\Cqr\p(0+)&= \gamma \Wq(0)-\frac{\phiqr\gamma}{\gamma+\delta}
	\end{align*}
	when $X$ is of bounded variation and
	\begin{align*}
	\Zqr\pp(0+) = -\gamma \Wq^{\prime}(0+)+\phiqr^2,\\
	\Cqr\pp(0+) = \gamma \Wq^{\prime}(0+)-\frac{\corr{\phiqr^2}\gamma}{\gamma+\delta}
	\end{align*}
	when $X$ is of unbounded variation.
	
	Therefore, 
	\begin{align*}
	V_{\kappa}'(b_u-;\pi_{b_u,b_l})=~&\frac{\gamma}{\gamma+\delta}\left( 1+\phiqr\left( d-\kappa+\frac{\mu}{\gamma+\delta} \right)
	-\left( d-\kappa \right)(\gamma+\delta) \Wq(0)\right) \nonumber\\
	&-\phiqr V_{\kappa}(b_u;\pi_{b_u,b_l})+\gamma \Wq(0)V_{\kappa}(b_u;\pi_{b_u,b_l})\nonumber\\
	&+\frac{\gamma}{\gamma+\delta}\phiqr V_{\kappa}(b_l;\pi_{b_u,b_l})-\gamma \Wq(0)V_{\kappa}(b_l;\pi_{b_u,b_l})\nonumber\\
	=~&V_{\kappa}'(b_u+;\pi_{b_u,b_l})+\gamma \Wq(0) \left( V_{\kappa}(b_u;\pi_{b_u,b_l})-V_{\kappa}(b_l;\pi_{b_u,b_l})-(d-\kappa)\right) 
	\end{align*}
	when $X$ is of bounded variation and
	\begin{align*}
	V_{\kappa}''(b_u-;\pi_{b_u,b_l})=V_{\kappa}''(b_u+;\pi_{b_u,b_l})-\gamma \Wq^{\prime}(0+) \left( V_{\kappa}(b_u;\pi_{b_u,b_l})-V_{\kappa}(b_l;\pi_{b_u,b_l})-(d-\kappa)\right) 
	\end{align*}
	when $X$ is of unbounded variation.
		Hence, \begin{equation*}
	V_{\kappa}(b_u;\pi_{b_u,b_l})-V_{\kappa}(b_l;\pi_{b_u,b_l})=d-\kappa=b_u-b_l-\kappa
	\end{equation*} 
	is the smoothness condition.
	
\end{proof}

\begin{remark}\label{meaning.smoothness}
	By rearranging (\ref{smoothness.condition}), we have 
	\begin{equation}\label{smooth.equivalent}
	V_{\kappa}(b_u;\pi_{b_u,b_l})=b_u-b_l-\kappa+V_{\kappa}(b_l;\pi_{b_u,b_l}).
	\end{equation}
This is equivalent to the continuity condition when dividends can be made at any time \citep*[see, e.g.][]{BaKyYa13,JeSh95,Loe08a}. In words, it means that the difference in value between both barriers $b_l$ and $b_u$ is exactly equal to the \emph{net} (of transaction costs $\kappa$) dividend paid between those two levels.

It is remarkable that this relation holds \emph{for any} smooth $(b_u,b_l)$ strategy in the periodic decision making framework considered in this paper.	
\end{remark}

\begin{remark}\label{remark.reduce.to.barrier}
	When the smoothness condition (\ref{smoothness.condition}) is met, one can further simplify the value function given in Theorem \ref{lemma.value.fcn} \corr{with the help of $V_\kappa(0;\pi_{b_u,b_l})=0$} and show that for $x\geq 0$,
	\begin{equation}\label{eqt.smooth}
	V_{\kappa}(x;\pi_{b_u,b_l}) = \frac{\Hqr(b_u)}{\Jqr(b_u)}\Jqr(b_u-x)-\Hqr(b_u-x)=V_{0}(x;\pi_{b_u}).
	\end{equation}

	Plugging  $x=b_u$ in (\ref{eqt.smooth}) and using $V_{\kappa}(0;\pi_{b_u,b_l}) =0$, we obtain
	\begin{equation}\label{Vbu.smooth}
	V_{\kappa}(b_u;\pi_{b_u,b_l}) =\frac{\Hqr(b_u)}{\Jqr(b_u)}+\frac{\gamma}{\gamma+\delta}\frac{\mu}{\delta}.
	\end{equation}
Furthermore, the last equality of (\ref{eqt.smooth}) implies that a smooth periodic $(b_u,b_l)$ strategy is increasing, i.e.
	 \begin{equation}\label{eqt.smoothbubl.Dgeq0}
	 V^\prime_{\kappa}(x;\pi_{b_u,b_l})=V^\prime_0(x;\pi_{b_u})\geq 0
	 \end{equation}
	 for all $x\geq 0$, thanks to (\ref{eqt.vprime.geq0}).
\end{remark}

Lemma \ref{lemma.smooth} characterised the condition for the value function of a $(b_u,b_l)$ strategy to be smooth. The following lemma characterises the smoothness condition explicitly in terms of $b_u$ and $b_l$.
\begin{lemma}\label{lemma.smoothness.equivalent}
	The smoothness condition (\ref{smoothness.condition}) is equivalent to $\Gamma_{b_l}(d)=0$, where $d=b_u-b_l$ and $\Gamma_{b_l}$ is defined as
	\begin{equation}\label{bubleqt}
	\Gamma_{b_l}(d):=\left(d-\kappa-\frac{\gamma}{\gamma+\delta}\frac{\mu}{\delta}\right) \Jqr(b_u)-\Hqr(b_u)+\Jqr(d)\Hqr(b_u)-\Jqr(b_u)\Hqr(d)=0.
	\end{equation}
\end{lemma}
\begin{proof}[Proof of Lemma \ref{lemma.smoothness.equivalent}]
	See Appendix \ref{App.C1}.
\end{proof}

	\begin{remark}
\corr{	The choice of notation $d=b_u-b_l$ reflects naturally the structure of the construction of the value function, which is first anchored at $b_l$ through a derivative. The distance between both barriers then depends directly on the level of fixed transaction costs, and has a direct interpretation as being the minimum viable amount of dividends to be paid. More rationale for this choice can be found in \citet[Remark A.3.1]{Tu17}, who revisited the original results of \citet{JeSh95}.}
	
	\end{remark}

The following proposition assures the existence of a periodic $(b_u,b_l)$ strategy that satisfies the smoothness condition (\ref{smoothness.condition}).

\begin{proposition}\label{lemma.existence.bu}
	For any \corr{$b_l\in[0,b^*]$}, there is a unique $b_u>b^*$ with $b_u> b_l+\kappa$ such that $\Gamma_{b_l}(b_u-b_l)=0$.
 \corr{Moreover, such $b_u$ is a continuous function of $b_l$. In particular, $b_u>b^*$ implies}
	\begin{equation}\label{eqt.smooth.dev.positive}
	\phiqr \frac{\Hqr(b_u)}{\Jqr(b_u)}+1>0.
	\end{equation}
\end{proposition}
\begin{proof}[Proof of Proposition \ref{lemma.existence.bu}]
	See Appendix \ref{App.C2}. 
\end{proof}

\begin{remark}\label{R_proofsmooth}
	When a liquidation-at-first-opportunity is considered, i.e. $b_l=0$, then the smoothness condition (\ref{smoothness.condition}) is equivalent to 
	\begin{equation}\label{eq.to.be.solve.bl.0}
	V_\kappa(b_u;\pi_{b_u,b_l})=\frac{\gamma}{\gamma+\delta}\frac{-\Aqr(b_u;d)}{\Zqr(b_u)}=b_u-\kappa.
	\end{equation}
	\corr{If the process survives until the next dividend decision time, then the net dividend payment will be \emph{at least} $b_u-\kappa$. So, $V_\kappa(b_u;\pi_{b_u,b_l})$ should be that quantity, discounted over the time to next dividend decision time where a dividend can be paid, times the} probability that this will happen (all in an expected sense). What this formula says is that smoothness ensures that it all balances out so as to obtain an expected present value of $b_u-\kappa$.
	
	This can be used to intuitively explain how Proposition \corr{\ref{lemma.existence.bu}} is proved. Consider the following two extreme cases. When $b_u$ is $\kappa$, the left hand side is a value function which is positive, which is greater than the right hand side which is zero. On the other hand when $b_u$ is large (close to infinity), the left hand side is $\frac{\gamma}{\gamma+\delta}(\text{``increase in surplus''}+b_u-\kappa)$, which is approximately $\frac{\gamma}{\gamma+\delta}(b_u-\kappa)$ since $b_u$ is large. This quantity is smaller than $b_u-\kappa$, the right hand side. In other words, since we are in periodic setting, we need to wait for the first opportunity to liquidate, the discounting effect is dominant when $b_u$ is large, resulting in the left hand side being smaller. Now it should be clear that there is a ``sweet spot'' such that the equation holds. 
	
	When $b_l>0$ a similar reasoning applies, although the proof is more involved; see Appendix \ref{Appendix.C}.	
\end{remark}

Thanks to Proposition \ref{lemma.existence.bu}, we know that for a given $b_l$, we can always find a $b_u$ ($>b_l+\kappa$) such that the smoothness condition (\ref{smoothness.condition}) is met. We call those strategies ``smooth $(b_u,b_l)$ strategy'' and denote those strategies as $\pis$ (the value function of the strategy is smooth). We should remember that $b_u$ is uniquely determined by $b_l$. In the following, we assume $b_u$, $b_l$ and $\kappa$ are fixed, i.e. we are looking at a particular smooth $(b_u,b_l)$ strategy. The value function of such strategy is denoted as \corr{$V_s(x)$} when the initial surplus is $x$. In the case where confusion may arise, we will write it explicitly as $V_\kappa(\cdot;\pis)$.

\subsection{Second step: a condition on the derivative of $V_s$ at $b_l$}\label{section.secondstep}

\corr{From last section, we know that for a fixed $b_l\in[0,b^*]$, we can always choose a unique $b_u>b_l+\kappa$ such that $V_\kappa(\cdot;\pi_{b_u,b_l})$ is smooth.} This is our first step to optimality, i.e. we shall only look at those strategies.

The second step to optimality concerns the derivative of the value function at $b_l$. Specifically, if $\Voo\leq 1$, we call the liquidation-at-first-opportunity strategy $\pi^{\kappa,s}_{b_u,0}$ ``optimal'' and denote it as $\pi^{\kappa,*}_{b_u,0}$. This means we choose $b_l=0=b_l^*$. On the other hand, if $V^\prime_\kappa(0;\pi^{\kappa,s}_{{b}_u,0})> 1$ and there are $(b_u,b_l)$ such that 
\begin{equation}\label{opt.condition}
V_\kappa^\prime(b_l;\pis)=1,
\end{equation}
then we also call it ``optimal'' and denote it as $\pio$. Hence, the notation $\pio$ with $b_l\in[0,b^*]$ stands for an ``optimal $(b_u,b_l)$ strategy''. The value function when a chosen optimal $(b_u,b_l)$ strategy $\pio$ (for $b_l\in[0,b^*]$) is applied is denoted as $V_*$, or $V_\kappa(\cdot,\pio)$ if the dependence of $b_l$ and $\kappa$ needs to be stressed.

Of course, our goal is to show that an ``optimal'' strategy exists and is optimal in the sense of (\ref{eqt.optimal}). The first goal (existence) is established \corr{at the end of this section after the following remarks} while the second goal (optimality) will be achieved in Section \ref{section.strategy.optimal}. 

\begin{remark}\label{remark.bubl.space}
	In the space of strategies being considered, it is useful to know the following relationship. 
	\begin{equation*}
	\mbox{	Set of optimal $(b_u,b_l)$ $\subseteq$ Set of smooth $(b_u,b_l)$ $\subseteq$ Set of general $(b_u,b_l)$}.
	\end{equation*}
	Ultimately, we want to select an element in the set of optimal $(b_u,b_l)$ strategy to serve our candidate strategy to be verified optimal. To do that, we need to make sure that there is at least one element in the first set. Specifically, Proposition \ref{lemma.existence.bu} ensures that the middle set has as many elements as the real line (and therefore is non-empty). Lemma \ref{lemma.bl.exist} (which appears later) guarantees that there is at least one element in the first set.	
\end{remark}

\begin{remark}\label{R_vdashkappa}
The derivative of the value function represents the marginal (discounted) rate of return from investing in the company. When the derivative is greater than 1, it means that an extra one dollar invested in the company will generate more than one dollar of return and therefore we should leave that dollar in the surplus, i.e. \corr{pay no dividends}. On the other hand, if the derivative is less than 1, the company cannot generate enough profit and hence it would be better off paying out the cash as dividends. Following this argument, we can see that if the company wants to maximise the total discounted dividends, it should pay dividends until it reaches a level where the derivative is $1$. This explains \eqref{opt.condition} and is a standard behaviour observed when (optimal) barriers are applied.

Furthermore, the condition $\Voo\leq 1$ represents the scenario that the company does not have a good prospects anywhere, so liquidating the company as soon as possible is desired, i.e. choose $b_l^*=0$. On the other hand, the condition $\Voo>1$ represents the scenario that the company has a good prospect and therefore we should keep the company running, i.e. choose $b_l^*>0$. In summary, we have
	\begin{equation}
	\begin{cases}
	\Voo\leq 1&\implies b_l^*=0\\
	\Voo>1 &\implies b_l^*>0
	\end{cases}.
	\end{equation}
\end{remark}
	
\begin{remark}\label{R_kappadistance}
Recall from Remark \ref{R_vdashkappa} that $V'$ can be interpreted as a marginal profitability rate. When $b_u$ and $b_l$ satisfy the smoothness condition (\ref{smoothness.condition}), by rearranging the terms we get
		\begin{equation}\label{E_intkappa}
		\kappa= \int_{b_l}^{b_u}(1-V_{\kappa}'(x;\pis))dx, 
		\end{equation}
which means that the distance between $b_l$ and $b_u$ must be such that the integrated ``profitability shortfall'' $1-V'$ is exactly $\kappa$.

Interestingly, \eqref{E_intkappa} always holds when dividends can be paid at any time \cite[see, e.g., Equation (3.13) of][]{JeSh95}. In our framework, this becomes a necessary (but not sufficient) condition for optimality. The condition becomes sufficient when $b_l$ is chosen so as to minimise the distance in \eqref{E_intkappa}, which is when \eqref{opt.condition} holds.
\end{remark}

\begin{remark}
		Recall from Remark \ref{remark.reduce.to.barrier} that the value function of a smooth periodic $(b_u,b_l)$ stategy simplifies, especially at $b_u$. The following displays the relationship between the optimal barriers $(b_u,b_l)$ versus the optimal barriers under different settings. For the optimal continuous barrier $\bar{b}>0$, it holds that $\Hqr(\bar{b})=0$, and the value at the barrier simplifies. For the optimal periodic barrier $b^*>0$ (without costs), we have $$\frac{\Hqr(b^*)}{\Jqr(b^*)}=-\frac{1}{\phiqr},$$
and further simplifications can be made. In our case however, we \corr{only have $V_\kappa'(b_l;\pi_{b_u,b_l})=1$ if $b_l>0$ which yields in view of \eqref{eqt.smooth} and \eqref{eqt.smooth.dev.positive}}
		\begin{equation}
		\frac{\Hqr(b_u)}{\Jqr(b_u)}=\frac{\Hqr^\prime(b_u-b_l)-1}{\Jqr^\prime(b_u-b_l)}>-\frac{1}{\phiqr} 
		\end{equation}
		when the smoothness condition is verified, and \eqref{eqt.smooth} does not seem to simplify further.

\end{remark}

\corr{We now proceed to prove the existence of $\pio$, i.e. the following proposition.
	\begin{lemma}\label{lemma.bl.exist}
		If $V_\kappa^\prime(0;\pi^{\kappa,s}_{{b}_u,0})>1$, then there exists $(b_u,b_l)$ with $b_l>0$ such that (\ref{opt.condition}) holds, i.e. $\pio$ exists.
	\end{lemma}
\begin{proof}[Proof of Lemma \ref{lemma.bl.exist}]
	From Proposition \ref{lemma.existence.bu}, we know that the mapping $b_l\mapsto V^\prime_\kappa(b_l;\pi_{b_u,b_l}^{\kappa,s})-1$ is a continuous function on the domain $b_l\in[0,b^*]$. The given condition in the lemma is the same as assuming the function value at $b_l=0$ is greater than $0$. Hence, we are done if we can show that the function value is negative at $b_l=b^*$. From \eqref{eqt.smooth}, such condition is the same as 
	$$V^\prime_0(b^*;\pi_{b^*+d(b^*)})<1$$
	with $d(b^*)$ being the unique solution to $\Gamma_{b^*}(d)=0$ (see Proposition \ref{lemma.existence.bu}). However, such condition is precisely Lemma \ref{Lemma.VdLeq1AtBstar}.
\end{proof}
}

Note that uniqueness will be shown in Lemma \ref{lemma.unique.bl}, whose proof uses results from Proposition \ref{thm} in the next section.

\corr{
	\begin{remark}
		Combining Lemmas \ref{lemma.bl.exist}, \ref{lemma.unique.bl} and Proposition \ref{lemma.existence.bu}, we conclude that there is a  unique pair $(b_u^*,b_l^*)$ with $0\leq b_l^*\leq b^*\leq b_u^*$ such that the periodic $(b_u,b_l)$ strategy $\pi_{b_u^*,b_l^*}$ is optimal.
	\end{remark}
}

\section{Verification of the optimality of the candidate $(b_u,b_l)$ strategy $\pio$}\label{section.strategy.optimal}

The previous section ensures that there is at least one $\pio$ (defined by (\ref{smoothness.condition}) and (\ref{opt.condition}), see Remark \ref{remark.bubl.space} for details). In this section, unless otherwise specified, we work on a chosen optimal $(b_u,b_l)$ strategy $\pio$ and show its optimality. Recall that the value function of such strategy is denoted as $\Vo$ or $V_\kappa(\cdot;\pio)$. Moreover, all properties derived for $\pis$ \corr{are} automatically satisfied by $\pio$. The optimality of a $\pio$ is summarised by the following theorem.

\begin{proposition}\label{thm}
	The strategy $\pio$ is optimal, i.e. $V_*(x)=V_\kappa(x;\pio)=v_{\kappa}(x)$ for $x\geq 0$.
\end{proposition}

We need some preparations for the proof. \corr{First, we establish Lemmas \ref{lemma.vdbuleq1}-\ref{lemma.8.6}.}

\begin{lemma}\label{lemma.vdbuleq1}
	$\Vs'(x)<1$ for $x\geq b_u$.
\end{lemma}
\begin{proof}[Proof of Lemma \ref{lemma.vdbuleq1}]
	For $x\geq b_u$, by substituting $\Vs(b_u)=\Vs(b_l)+d-\kappa$ in (\ref{V.kapparho.upper}), we have 
	\begin{equation}
	\Vs(x)=\frac{\gamma}{\gamma+\delta}\left(x-b_u+\frac{\mu}{\gamma+\delta}(1-e^{-\phiqr(x-b_u)})\right)+\Vs(b_u)\left(\frac{\gamma}{\gamma+\delta}+\frac{\delta}{\gamma+\delta}e^{-\phiqr(x-b_u)}\right). \label{Value.upper.opt}
	\end{equation}
	By taking derivative w.r.t. $x$ in (\ref{Value.upper.opt}) and using $\Vs(b_u)=\frac{\Hqr(b_u)}{\Jqr(b_u)}+\frac{\gamma}{\gamma+\delta}\frac{\mu}{\delta}$, we have
	\begin{align}
	\Vs'(x)=~&\frac{\gamma}{\gamma+\delta}+\frac{\gamma}{\gamma+\delta}\frac{\mu}{\gamma+\delta}\phiqr e^{-\phiqr(x-b_u)}+\Vs(b_u)\left(-\frac{\delta}{\gamma+\delta}\phiqr e^{-\phiqr(x-b_u)}\right)\nonumber\\
	=~&\frac{\gamma}{\gamma+\delta}+\frac{\delta}{\gamma+\delta}\phiqr e^{-\phiqr(x-b_u)} \left(\corr{\frac{\gamma}{\gamma+\delta}}\frac{\mu}{\delta}-\Vs(b_u)\right)\nonumber\\
	=~&\frac{\gamma}{\gamma+\delta}-\frac{\delta}{\gamma+\delta} e^{-\phiqr(x-b_u)}\phiqr\frac{\Hqr(b_u)}{\Jqr(b_u)}\label{Vtop.dev.1}\\
	=~&\frac{\gamma}{\gamma+\delta}+\frac{\delta}{\gamma+\delta}e^{-\phiqr(x-b_u)}-\left(\frac{\delta}{\gamma+\delta} \left(\phiqr\frac{\Hqr(b_u)}{\Jqr(b_u)}+1\right)\right)e^{-\phiqr(x-b_u)}\label{Vtop.dev.2}.
	\end{align}
	If $\Hqr(b_u)\geq 0$, from (\ref{Vtop.dev.1}) we have $\Vs'(x)\leq\frac{\gamma}{\gamma+\delta}<1$. On the other hand, if $\Hqr(b_u)<0$, from (\ref{Vtop.dev.2}), we have $\Vs'(x)<\frac{\gamma}{\gamma+\delta}+\frac{\delta}{\gamma+\delta}e^{-\phiqr(x-b_u)}\leq\frac{\gamma}{\gamma+\delta}+\frac{\delta}{\gamma+\delta}=1$ since $\phiqr\frac{\Hqr(b_u)}{\Jqr(b_u)}+1>0$ by (\ref{eqt.smooth.dev.positive}).
\end{proof}

\begin{lemma}\label{lemma.one.turning.point}
	The derivative of the value function, $\Vs'$, has at most one turning point on $[0,b_u]$.
\end{lemma}
\begin{proof}[Proof of Lemma \ref{lemma.one.turning.point}]
	Note that $\Jqr\p(x)=\frac{\delta}{\gamma+\delta}\phiqr \Zqr(x)$. In addition, from (\ref{eqt.smooth}), we have
	\begin{equation}
	\Vs(x)= \frac{\Hqr(b_u)}{\Jqr(b_u)}\Jqr(b_u-x)-\Hqr(b_u-x).
	\end{equation}	
	Hence, we have
	\begin{align}
	\Vs''(x)&=\frac{\Hqr(b_u)}{\Jqr(b_u)}\Jqr\pp(b_u-x)-\Hqr\pp(b_u-x)\nonumber\\
	&=\frac{\Hqr(b_u)}{\Jqr(b_u)}\phiqr\frac{\delta}{\gamma+\delta}(\phiqr \Zqr(b_u-x)-\gamma \Wq(b_u-x))-\frac{\gamma}{\gamma+\delta}\delta \Wq(b_u-x)\nonumber\\
	&=\frac{\delta}{\gamma+\delta}\left(\left(\frac{\Hqr(b_u)}{\Jqr(b_u)}\phiqr\right)\phiqr \Zqr(b_u-x)-\left(\frac{\Hqr(b_u)}{\Jqr(b_u)}\phiqr+1\right)\gamma \Wq(b_u-x)\right)\label{vdd.hleq0}\\
	&=\frac{\delta}{\gamma+\delta}\left(\frac{\Hqr(b_u)}{\Jqr(b_u)}\phiqr+1\right)\gamma \Wq(b_u-x)\left(\frac{\frac{\Hqr(b_u)}{\Jqr(b_u)}\phiqr}{\frac{\Hqr(b_u)}{\Jqr(b_u)}\phiqr+1}\frac{\phiqr}{\gamma}\frac{\Zqr(b_u-x)}{\Wq(b_u-x)}-1\right).\label{vdd.hg0}
	\end{align}
	If $\Hqr(b_u)<0$, we have from (\ref{vdd.hleq0}) that $\Vs''(x)<0$ for all $x\in[0,b_u]$ since $\frac{\Hqr(b_u)}{\Jqr(b_u)}\phiqr+1> 0$ by (\ref{eqt.smooth.dev.positive}), which means that there is no turning point in this case. On the other hand, if $\Hqr(b_u)\geq 0$, we have from (\ref{vdd.hg0})
	\begin{equation*}
	\Vs''(x)\geq 0 \iff \frac{\frac{\Hqr(b_u)}{\Jqr(b_u)}\phiqr}{\frac{\Hqr(b_u)}{\Jqr(b_u)}\phiqr+1}\frac{\phiqr}{\gamma}\frac{\Zqr(b_u-x)}{\Wq(b_u-x)}-1\geq 0,~x\in[0,b_u]. 
	\end{equation*}
	Noting from (\ref{ineqt.Z.W}) that $\Zqr(b_u-x)/\Wq(b_u-x)$ is (strictly) increasing in $x$, we can conclude that there is at most one turning point for $\Vs'$.
\end{proof}

\begin{lemma}\label{lemma.concave.bl}
	If $b_l>0$ and $\Vs'(b_l)\geq1$, then $\Vs''(x)<0$, $x\in[0,b_l]$.
\end{lemma}
\begin{proof}[Proof of Lemma \ref{lemma.concave.bl}]
	From the proof of Lemma \ref{lemma.one.turning.point}, we know that $\Vs''(x)$ cannot go from positive to negative when $x$ increases. Hence, we can conclude that there are only \corr{three} possibilities regarding the sign of $\Vs''$ on $[0,b_u]$:
	\begin{enumerate}
		\item $\Vs''(x)\geq 0$ for all $x\in[0,b_u]$;
		\item There exists a point $x_t\in(0,b_u)$ such that 
		\begin{equation*}
		\begin{cases}
		\Vs''(x)<0,~x\in[0,x_t)\\
		\Vs''(x_t)=0\\
		\Vs''(x)>0,~x\in(x_t,b_u]
		\end{cases};
		\end{equation*} 
		\item $\Vs''(x)\leq 0$ for all $x\in[0,b_u]$.
	\end{enumerate}
	Case (1) is impossible since it implies \corr{$1\leq \Vs'(b_l)\leq \Vs'(b_u)$}, which contradicts to Lemma \ref{lemma.vdbuleq1}. For case (2), unless $x_t>b_l$ we can use the same argument as in Case (1) to conclude that it is impossible. Therefore, we have $\Vs''(x)\leq 0$ for $x\in[0,b_l]$. Case (3) directly leads to $\Vs''(x)\leq 0$ for $x\in[0,b_l]$.
\end{proof}

\begin{lemma}\label{lemma.8.6}
	For $x\geq\kappa$, we have
	\begin{equation}
	(\mathcal{L}-\delta)\Vo(x)+\gamma \left( \max_{\kappa\leq l \leq x} \left\{ (l-\kappa) + \Vo(x-l) -\Vo(x)\right\}\right)_+= 0.
	\end{equation}
\end{lemma}
\begin{proof}[Proof of Lemma \ref{lemma.8.6}]
	From equations (4.20), (4.21), (4.18) and (4.22) in \citet*{PeYa16}, we have
	\begin{align}
	(\mathcal{L}-\delta)\Jqr(x)&=0,\\
	(\mathcal{L}-\delta)\Hqr(x)&=0,\\
	(\mathcal{L}-\delta)(x-b_u)&=\mu-\delta(x-b_u),\\
	(\mathcal{L}-\delta)e^{-\phiqr (x-b_u)}&=\gamma e^{-\phiqr (x-b_u)}.
	\end{align}
	By following exactly the same steps, it can be shown that
	\begin{align}
	(\mathcal{L}-\delta)\Vo(x)&=0, &x\leq b_u,\label{eqt.int.dif.bu.bl.lower}\\
	(\mathcal{L}-\delta)\Vo(x)+\gamma\big[x-b_l-\kappa+\Vo(b_l)-\Vo(x)\big]&=0, &x> b_u.\label{eqt.int.dif.bu.bl.upper}
	\end{align}
	
	Next, from the definition of $\pio$ \corr{and Lemma \ref{lemma.vdbuleq1}}, when $b_l=0$, we have that $\Vo'(x)<1$ for $x\in(0,b_u]$ in all 3 cases stated in the proof of Lemma \ref{lemma.concave.bl} regarding the sign of $\Vo''(x)$. When $b_l>0$, only cases 2 and 3 are possible. Hence, we have $\Vo'(x)<1$ for $x\in(b_l,b_u]$ by Lemma \ref{lemma.concave.bl}. Combining with Lemma \ref{lemma.vdbuleq1}, we have 
	\begin{equation}\label{verify.eqt.1}
	\begin{cases}
	\Vo'(x)\leq 1,~x\geq 0,~&\mbox{if }b_l=0,\\
	\Vo'(x)\begin{cases}
	&>1,~x<b_l\\
	&=1,~x=b_l\\
	&<1,~x>b_l
	\end{cases},~&\mbox{if }b_l>0.
	\end{cases}
	\end{equation}
	and subsequently
	\begin{equation}\label{verify.eqt.2}
	\begin{cases} x-b_l-\kappa + \Vo(b_l)- \Vo(x)<0, &x \in [b_l,b_u)\\
	b_u-b_l-\kappa + \Vo(b_l) - \Vo(b_u)=0  \\
	x-b_l-\kappa +\Vo(b_l) -\Vo(x)>0, &x\in (b_u,\infty)\end{cases}.
	\end{equation}
	
	Next, we verify that for a fixed $x$, $l=(x-b_l)1_{\{x\geq b_u\}}$ maximises 
	\begin{equation}
	P(l):={(l-\kappa)}_+ + \Vo(x-l) - \Vo(x),~x\geq\kappa,
	\end{equation}
	with the support of $P$ being $\{0\}\cup[\kappa,x]$. 
	
	First, we consider the support $P$ on $[\kappa,x]$. Since $P$ is a continuous differentiable function (as $\Vo$ is) and the support is bounded, the maximum value of the function is attended at either $P'=0$ or at the boundaries. Now, by taking the derivative of $P$, we have
	\begin{align}
	P'(l)&=1-\Vo'(x-l)\nonumber\\
	&\begin{cases}
	<0,~~l>x-b_l\\
	=0,~~l=x-b_l\\
	>0,~~l<x-b_l
	\end{cases}
	\end{align}
	thanks to (\ref{verify.eqt.1}).
	Hence, the maximum value of $P$ on $[\kappa,x]$ is attained at $l^*=(x-b_l)1_{\{x-b_l\geq\kappa\}}+\kappa1_{\{x-b_l<\kappa\}}$ since $P$ is strictly decreasing on $l>x-b_l$. Now, we should compare the value of $P$ at $0$ and $l^*$ to find the maximum. Clearly, $P(0)=0$. When $x-b_l<\kappa$, $P(l^*)=P(\kappa)=\Vo(x-\kappa) - \Vo(x)\leq0=P(0)$, since $\Vo'(x)\geq0$ for all $x\geq 0$. Thus, the maximum value attains at $l=0$. On the other hand, if $x-b_l\geq\kappa$, then $P(l^*)=x-b_l-\kappa+\Vo(b_l)- \Vo(x)\geq 0=\corr{P(0)}$ if and only if $x\geq b_u$ by (\ref{verify.eqt.2}). 
	In summary, $P$ attains its maximum when $l=0$ if $x<b_u$, and $l=x-b_l$ if $x\geq b_u$.
	
	Therefore, for $x\geq\kappa$, we have
	\begin{align}
	&\left( \max_{\kappa\leq l \leq x} \left\{ (l-\kappa)  + \Vo(x-l) - \Vo(x)\right\}\right)_+\nonumber\\
	=&\begin{cases} \Vo(x)-\Vo(x), &x \in [\kappa, b_u) \\
	x-b_l-\kappa+\Vo(b_l)-\Vo(x),&x \in [b_u, \infty) \end{cases} \nonumber\\
	=&\begin{cases} 0,&x \in [\kappa, b_u) \\
	x-b_l-\kappa+\Vo(b_l)-\Vo(x),&x \in [b_u, \infty) \end{cases}
	\end{align}
	and hence 
	\begin{equation}\label{eqt.conditon.2.3}
	(\mathcal{L}-\delta)\Vo(x)+\gamma \left( \max_{\kappa\leq l \leq x} \left\{ (l-\kappa) + \Vo(x-l) -\Vo(x)\right\}\right)_+= 0,~x\geq\kappa
	\end{equation}
	by (\ref{eqt.int.dif.bu.bl.lower}) and (\ref{eqt.int.dif.bu.bl.upper}). Note that $(\mathcal{L}-\delta)\Vo(b_u)$ is  well defined as $\Vo(x)$ is smooth.
\end{proof}

We are now ready to prove Proposition \ref{thm}.

\begin{proof}[Proof of Proposition \ref{thm}]
	It suffices to show that $\Vo$ satisfies all 3 conditions in Lemma \ref{verification.lemma}.
	
	Since both \corr{$\Hqr^\prime(0)$ and $\Jqr^\prime(0)$} are finite, combining with (\ref{eqt.smoothbubl.Dgeq0}) and Lemma \ref{lemma.vdbuleq1}, we have
	\begin{equation}\label{der.bdd}
	0\leq \Vo'(x)< \Vo'(0)+1,~x\geq 0,
	\end{equation}
	showing that $\Vo'$ is finite. This together with Proposition \ref{lemma.existence.bu} shows the first condition. The second condition is satisfied by (\ref{der.bdd}) \corr{together with the fact that $V_*(0)=0$ ; see the discussion after \eqref{def.value.strategy}}. The third condition is a simple consequence of (\ref{eqt.conditon.2.3}) and (\ref{eqt.int.dif.bu.bl.lower}). 
\end{proof}

From Lemma \ref{lemma.bl.exist} and Proposition \ref{thm}, we know that there exists a pair of ($b_u,b_l$) such that $\pi_{b_u,b_l}=\pio$ is optimal. The following lemma states that the choice of ($b_u,b_l$) is unique.

\begin{lemma}\label{lemma.unique.bl}
	There is only one pair of $(b_u,b_l)$ such that the strategy $\pi_{b_u,b_l}$ qualifies as $\pio$.
\end{lemma}

\begin{proof}[Proof of Lemma \ref{lemma.unique.bl}]
	If $b_l=0$, clearly it is unique. Otherwise, suppose we have \corr{two} $\pio$ strategies, namely $\pioo$ and $\piot$. From Proposition \ref{thm}, we have
	\begin{equation}
	V_\kappa(x;\pioo)=V_\kappa(x;\piot)=v_{\kappa}(x).
	\end{equation}
	As a result, by using the definition of $\pioo$ and $\piot$, we have
	\begin{equation}
	V_\kappa^\prime(b_l^{(1)};\pioo)=V_\kappa^\prime(b_l^{(2)};\piot)=1.
	\end{equation}
	
	From equation (\ref{verify.eqt.1}), we have $V_\kappa'(x;\pioo)<1$ for $x>b_l^{(1)}$, which gives $b_l^{(2)}\leq b_l^{(1)}$. Similarly, we have $b_l^{(1)}\leq b_l^{(2)}$. Therefore, we have $b_l^{(1)}= b_l^{(2)}$. Recall from Proposition \ref{lemma.existence.bu} that for a given $b_l$ there is a unique $b_u$ to achieve smoothness.. Hence, $b_l^{(1)}= b_l^{(2)}$ implies $b_u^{(1)}= b_u^{(2)}$. In other words, there is only one $\pio$.

\end{proof}

We are now ready to state our main result:

\begin{theorem}
	Denote $(b_u^*,b_l^*)$ the barriers of \corr{the unique} $\pio$, 
	the strategy $\pi_{b_u^*,b_l^*}$ is optimal, i.e.
	\[ V_{\kappa}(x;\pi_{b_u^*,b_l^*}) = v_{\kappa}(x) ~\mbox{for all}~ x\geq 0.\]
\end{theorem}

\begin{proof}
Proposition \ref{thm} states that the family of optimal $(b_u,b_l)$ strategies is optimal. Lemma \ref{lemma.bl.exist} states that the family of optimal $(b_u,b_l)$ strategies has at least one element while Lemma \ref{lemma.unique.bl} states that the family of optimal $(b_u,b_l)$ strategies has at most one element. All together, it means there exists a unique (optimal) periodic $(b_u,b_l)$ strategy which is optimal in the sense of (\ref{eqt.optimal}). 
\end{proof}

\section{Numerical illustrations}\label{section.illustration}

In this section, a diffusion process with Poissonian upward exponential jumps is used, i.e.
\begin{equation}
X(t) = x-ct+\sigma W(t)+\sum_{i=1}^{N(t)}G_i,
\end{equation}
where $\{G_i,i\in\mathbb{N}\}$ is a collection of i.i.d. exponential random variables with mean $1/\beta$, $\{W(t);t\geq 0\}$ is a standard diffusion process and $\{N(t):t\geq 0\}$ is a Poisson process with rate $\lambda$ such that $\mathbb{E}(N(t))=\lambda t$.  The baseline parameters used are $c=0.0027$, $\sigma=0.09$, $\lambda=1$, $\beta=33.33$, $\gamma=0.04$, $\delta=0.003$ and $\kappa=0.06$. 

In the terminology of (\ref{def.snlp.2}), the Laplace exponent in our illustration is
$$\psi_Y(\theta)=c\theta+\frac{\sigma^2}{2}\theta^2+\int_{(0,\infty)}(e^{-\theta}-1)\lambda(\beta e^{-\beta s})ds,$$
which is slightly different but can be rewritten in the form of (\ref{def.snlp.2}) easily. This is further explicitly evaluated as 
$$\psi_Y(\theta)-q=c\theta+\frac{\sigma^2}{2}\theta^2+\lambda\frac{\beta}{\beta+\theta}-\gamma-q.$$
It is easy to show that $\psi_Y$ is a rational function with 3 distinct roots and therefore its reciprocal can be rewritten using partial fraction as 
$$\frac{1}{\psi_Y(\theta)-q}=\sum_{j=1}^3\frac{1}{\psi_Y^\prime(r_j^{(q)})}\frac{1}{\theta-r_j^{(q)}},$$
where $r_j^{(q)}$, $j=1,2,3$ are the roots of $\psi_Y(\theta)-q=0$.
The $q$-scale function $W_q$ can then be computed explicitly by inverting the Laplace transform. All other scale functions can then be computed explicitly afterwards.

To find the optimal barriers $(b_u^*,b_l^*)$, we make use of \corr{Proposition} \ref{lemma.existence.bu} and Lemma \ref{lemma.bl.exist}. To be more specific, we perform the following:
\begin{enumerate}
	\item Find $b^*$ using (\ref{eqt.Q}). Specifically, if $Q(0)\geq 0$, then set $b^*=0$, otherwise, solve $b^*$ such that \corr{$Q(b^*)= 0$}. This can be done by (1) trying a large enough $b$ such that $Q(b)> 0$ following by (2) a bisection method on the range $[0,b]$.
	\item Write a function on $b_l\geq 0$ to output $b_u$ from \corr{Proposition} \ref{lemma.existence.bu} with a similar method as the previous step (using range $[\max(\kappa,b^*),b]$ for large enough $b$), then calculate the derivative of the value function at $b_l$ and return this number. Say we call this function $G$.
	\item Find $b_l^*$ using Lemma \ref{lemma.bl.exist}. Specifically, if $G(0)\leq 1$, then we set $b_l^*=0$, otherwise we can obtain $b_l^*$ by solving $G(b_l^*)=1$ via a bisection method on the range $[0,\corr{b^*}]$. Use Proposition \ref{lemma.existence.bu} to calculate $b_u^*$ from $b_l^*$.
\end{enumerate}

\begin{remark}
	We remark that gradient descend type of methods typically do not work well because a realatively large increment of the parameters (barriers) only results in a small change of the objective function (i.e. plateau). Therefore, analytic methods (such as used in this paper) are needed. Perhaps more importantly, this shows that in practice one typically have more flexibility to deviate from the optimal strategy to incorporate other considerations.
\end{remark}

\begin{remark}
	Note that $\gamma$, $\delta$ and $\lambda$ are forces of dividend decisions, continuous interest, and gain occurrence, per unit of time, respectively. Therefore, the value of $\gamma$ needs to be compared with those of $\lambda$ and $\delta$. Similarly, the value of $\kappa$ is in currency unit. Along those lines, it needs to be commensurate with those of $c$, $\sigma$ and $1/\beta$. 
\end{remark}

\begin{remark}	
	The numerical values chosen are inspired by the following fictitious business. A real estate business which on average sells 50 houses a year and pays biannual dividend, i.e. $\lambda=1$ and $\gamma=2/50=0.04$ and the time unit is (roughly) a week. Hence, $\delta=0.003$ implies an annual force of interest of $15\%$. In addition, for each house sold, the commission gained is on average $0.03$ unit, i.e. $\beta=33.33$. (For instance, typical commission rates in Sydney are $2\%$ and the median house price is about \$1,150,000, so that $1/\beta$ would be $\$23,000/\ln 2\approx\$33,000$ or 0.03 million) Furthermore, to illustrate the riskiness of the business, we assume $c=0.027$ and $\sigma=0.09$ such that the cost of the business is $90\%$ of its expected gain. 
	Lastly, the size of $\kappa$ is assumed to be $0.06$, approximately 2 weeks of cost. 
	

\end{remark}

Illustrations include the impact,  on the optimal dividend strategy, of the transaction costs and the interplay of dividend decision frequency and force of interest. 

\subsection{Impact of the fixed transaction costs $\kappa$}

\begin{figure}[H]
	\centering
	\includegraphics[width = 0.6\textwidth]{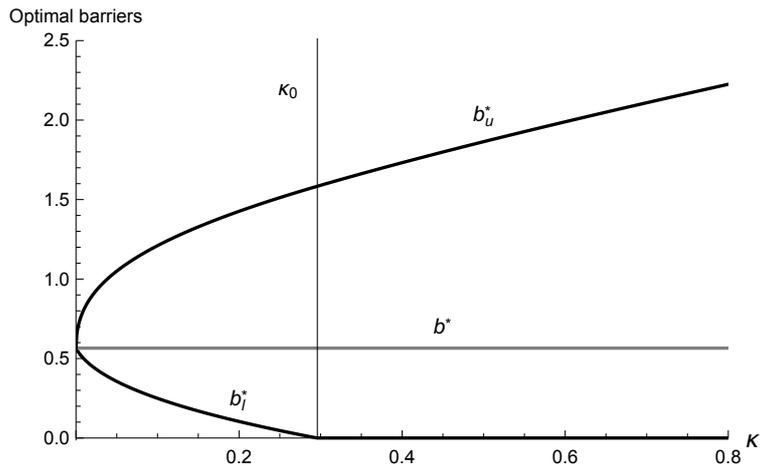}
	\caption{Impact of the fixed transaction costs }	\label{Fig.kappa} 
\end{figure}

Figure \ref{Fig.kappa} plots the optimal barriers $b_u^*$ and $b_l^*$ against fixed transaction costs $\kappa$. We also denote $b^*$ as the optimal barrier when there is no transaction costs, see e.g. \citet*{PeYa16} for details. This figure illustrates that when the transaction cost $\kappa$ increases from $0$, the optimal barrier without transaction costs $b^*$ splits into the upper and the lower barrier $b_u^*$ and $b_l^*$, respectively. As $\kappa$ further increases, the distance between both barriers increases, too; see also Remark \ref{R_kappadistance}. When the transaction costs are more than a certain quantity which we denote $\kappa_0$, then $b_l^*=0$ and a liquidation at first opportunity is optimal. This illustrates that despite the profitability ($\mu$), the business would not have good prospect if the costs of paying the shareholders ($\kappa$) are too high. Lemma \ref{lemma.kappa0.exist} establishes (partially) the existence of $\kappa_0$. 
\begin{lemma}\label{lemma.kappa0.exist}
	There exists a threshold $\kappa_0\in\mathbb{R}^+\cup\{\infty\}$ such that $\kappa\geq\kappa_0\iff b_l^*=0$.
\end{lemma}
\begin{proof}[Proof of Lemma \ref{lemma.kappa0.exist}]
	See Appendix \ref{Appendix.D}.
\end{proof}

\subsection{Impact of the time parameters $\gamma$ and $\delta$}
Figure \ref{fig_sensitivity_delta} plots the optimal barriers $b_u^*$, $b_l^*$ and $b^*$ (the optimal periodic barrier when $\kappa=0$) against the frequency parameter of dividend decisions $\gamma$. By looking at the graphs from left to right, we can see the convergence of $b_u$, $b_l$ and $b^*$ to their ``continuous'' counterparts (that is, the optimal impulse strategy when dividends can be paid at any time; see references below), as indicated by the 3 horizontal lines. These limits can be calculated, for example, using the results from \citet*[without and with transaction costs, respectively]{BaKyYa12,BaKyYa13}. The expected present value of dividends under $b^*$ (in this case with $\kappa=0$) can itself be calculated using the results from \citet*{PeYa16}. The curves for $b_u^*$ and $b_l^*$ developed in this paper are calculated using Proposition \ref{lemma.existence.bu} and Lemma \ref{lemma.bl.exist}, respectively.
\begin{figure}[H]
	\centering
	\subfigure[$\delta=0.003$]{\includegraphics[width=0.31\textwidth]{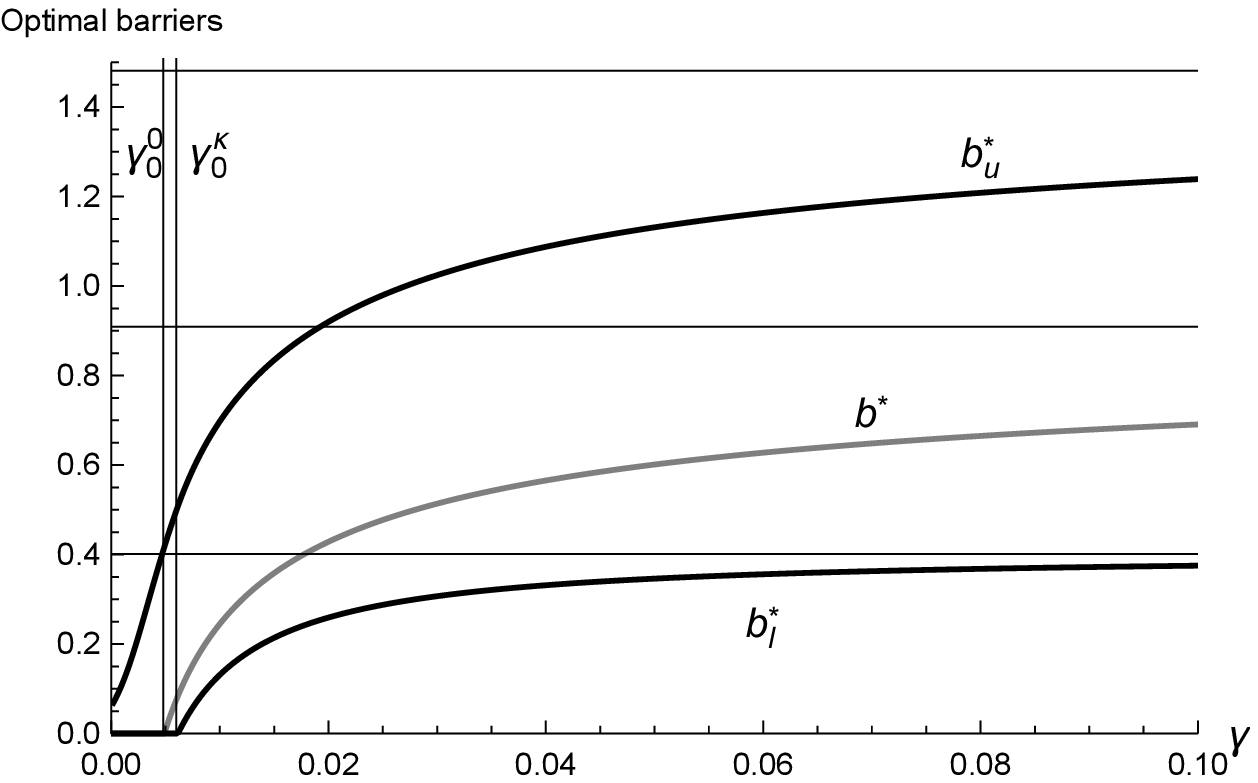}}
	\subfigure[$\delta=0.004$]{\includegraphics[width=0.31\textwidth]{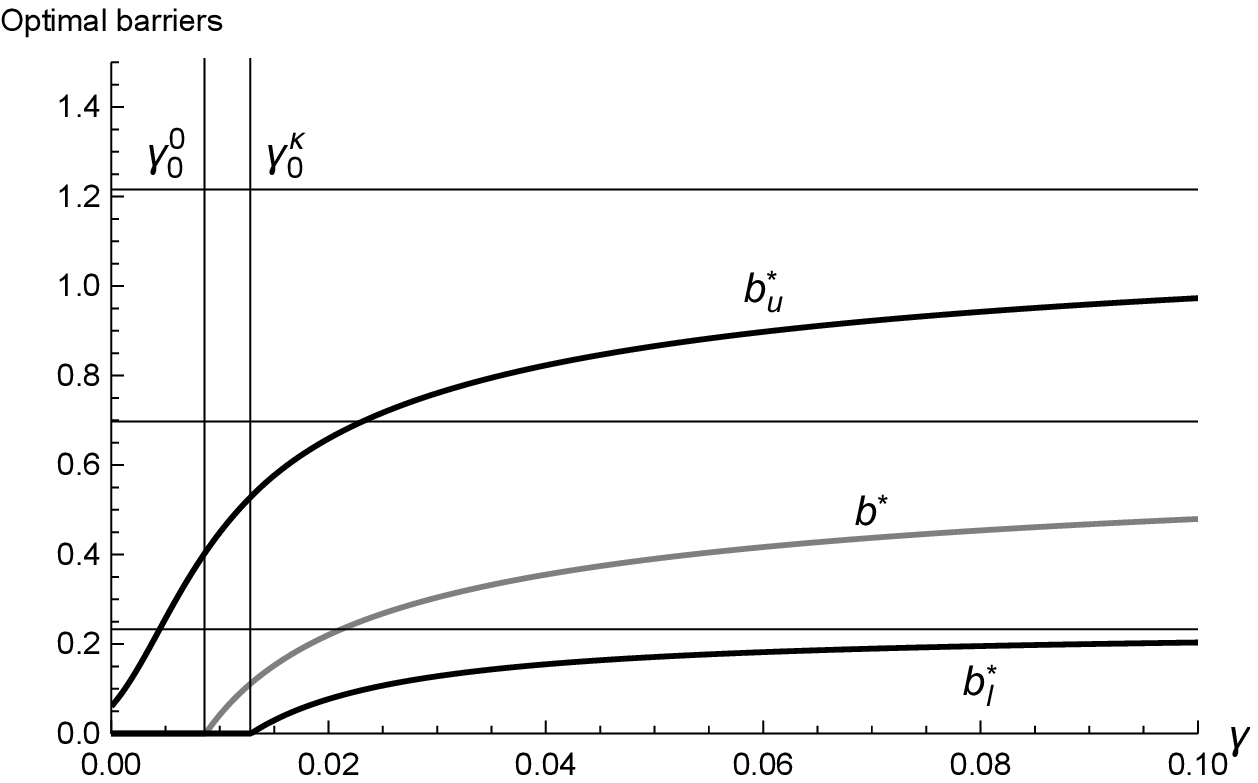}}
	\subfigure[$\delta=0.005$]{\includegraphics[width=0.31\textwidth]{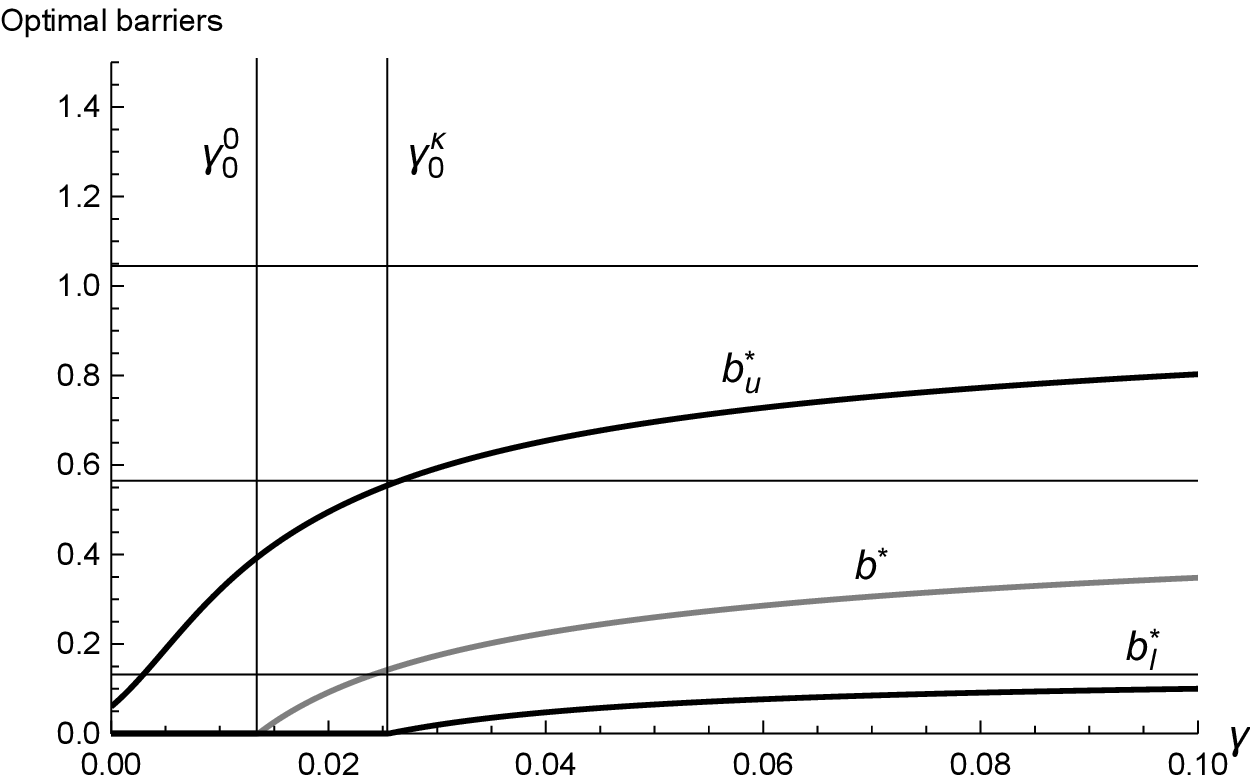}}
	\caption{Sensitivity to time parameters $\gamma$ and $\delta$} 
	\label{fig_sensitivity_delta}
\end{figure}

Again, one can see that the  barrier $b^*$ is sandwiched between both levels of the $(b_u^*,b_l^*)$ strategy. Furthermore, as the `time impatience' parameter $\delta$ increases, it becomes more important to pay more dividends earlier (as compared to avoid ruin), and the barrier levels decrease. However, at the same time, the intensity at which dividends decisions are made $\gamma$ (in a way, how often they can be paid) also need to be higher, lest a liquidation-at-first-opportunity strategy becomes optimal. Indeed, we can see that there is a threshold for $\gamma$ such that below this threshold a liquidation-at-first-opportunity strategy is optimal, that is, $b_l^*=0$. This threshold $\gamma_0^\kappa$ (in presence of fixed transaction costs $\kappa$)  is obtained in a similar way to $\kappa_0$ (see Lemma \ref{lemma.kappa0.exist}). Obviously, the introduction of transaction costs $\kappa=0.6$ pushes this threshold upwards, which explains why $b^*$ leaves the $\gamma$ axis earlier (at $\gamma_0^0$).

Generally speaking, when $b_l^*=0$ (when a liquidation-at-first-opportunity strategy is optimal), the difference $b_u-b_l=b_u$ is strictly larger than the transaction costs $\kappa$ to allow for a strictly positive final dividend (upon liquidation) even when the current surplus is below $\kappa$ (not enough to pay the liquidation cost); see Appendix \ref{Appendix.C} for details. However, when $\gamma$ becomes small (and one will have to wait longer to be able to liquidate), one would give up the buffer to exchange for a higher chance of exiting the business. Eventually, when $\gamma$ tends to $0$, $b_u^*$ decreases to exactly $\kappa~(=0.06)$ so that $b_u^*$ intersects the  $y$-axis at $\kappa$.

Furthermore, it is interesting to note that  the `periodic' $(b_u^*,b_l^*)$ is strictly below the split barrier (impulse) strategy of the `continuous' case (when dividends can be paid at any time)---indicated with horizontal gray barriers in Figure \ref{fig_sensitivity_delta}. In the case of $b_u^*$ this is because the strategy compensates for the cost of having to wait another period if dividends are not paid immediately. This difference is larger (and convergence slower) as $\delta$ increases, which makes sense. The convergence of $b_l^*$ seems to be quicker than that of $b_u^*$, simply because in its case the danger associated with being too close to 0 is likely to overpower the force described earlier in this paragraph (and which would push it down).

\section{Concluding remarks}\label{section.conclusion}
In this paper, we determined the form of the optimal periodic dividend strategy when there are fixed transaction costs, when the dividend decisions are Poissonian, and where  the underlying model is a spectrally positive L\'evy process. Using exiting identities and the strong Markov properties, we were able to compute the value function of a periodic $(b_u,b_l)$ strategy concisely in terms of scale functions.

We proceeded to identify the best strategy among the class of periodic $(b_u,b_l)$ strategies, namely the periodic $(b_u^*,b_l^*)$ strategy, in 2 steps and verified its optimality. A number of new insights were gained while doing so. In particular, the difference in the barriers $b_u-b_l$ is always strictly greater than the transaction cost $\kappa$ such that the net dividend is at least $b_u-b_l-\kappa>0$, i.e. a buffer. Moreover, despite the profitability $\mu$, when the transaction costs $\kappa$ are too high, it is optimal to close the business.

Finally, we numerically illustrated the convergence of our results with that of \citet*{BaKyYa13} and \citet*{PeYa16}, as well as the impact of the transaction costs and the frequency of dividend decisions on the optimal barriers.

This paper is a significant step towards answering a number of open questions, including: (i) can hybrid (periodic and continuous) dividend strategies \citep*[see][]{AvTuWo16} be optimal in presence of fixed transaction costs, and if so, under what conditions? (ii) is the optimal periodic dividend strategy still a $(b_u,b_l)$ strategy when inter-dividend decision times are Erlang$(n)$ distributed, $n\ge2$? (iii) are $(b_u,b_l)$ strategies also optimal in spectrally \emph{negative} L\'evy risk processes?

\section*{Acknowledgments}

\corr{The authors are indebted to two anonymous referees, whose comments led to significant improvements of the manuscript.}

This paper was presented at the Australasian Actuarial Education and Research Symposium in December 2017 in Sydney (Australia), at the 10th Conference in Actuarial Science \& Finance on Samos (Greece) in June 2018, at the $22^\text{nd}$ International Congress on Insurance: Mathematics and Economics (Sydney, Australia) in July 2018, and at the $4^\text{th}$ European Actuarial Journal Conference (Leuven, Belgium) in September 2018. The authors are grateful for constructive comments received from colleagues who provided constructive comments on the paper.

This research was  supported under Australian Research Council's Linkage \corrs{(LP130100723) and Discovery (DP200101859) Projects funding schemes.} Hayden Lau acknowledges financial support from an Australian Postgraduate Award and supplementary scholarships provided by the UNSW Australia Business School. The views expressed herein are those of the authors and are not necessarily those of the supporting organisations. 

\section*{References}

\bibliographystyle{elsarticle-harv}

\bibliography{libraries2} 

\begin{thebibliography}{39}
\expandafter\ifx\csname natexlab\endcsname\relax\def\natexlab#1{#1}\fi
\expandafter\ifx\csname url\endcsname\relax
  \def\url#1{\texttt{#1}}\fi
\expandafter\ifx\csname urlprefix\endcsname\relax\def\urlprefix{URL }\fi

\bibitem[{Albrecher et~al.(2011)Albrecher, Cheung, and Thonhauser}]{AlChTh11a}
Albrecher, H., Cheung, E. C.~K., Thonhauser, S., 2011. Randomized observation
  periods for the compound poisson risk model: dividends. {ASTIN} Bulletin
  41~(2), 645--672.

\bibitem[{Albrecher et~al.(2016)Albrecher, Ivanovs, and Zhou}]{AlIvZh16}
Albrecher, H., Ivanovs, J., Zhou, X., 2016. Exit identities for {L}\'evy
  processes observed at poisson arrival times. Bernoulli 22, 1364--1382.

\bibitem[{Albrecher and Thonhauser(2009)}]{AlTh09}
Albrecher, H., Thonhauser, S., 2009. Optimality results for dividend problems
  in insurance. RACSAM Revista de la Real Academia de Ciencias; Serie A,
  Mathem{\'a}ticas 100~(2), 295--320.

\bibitem[{Avanzi et~al.(2013)Avanzi, Cheung, Wong, and Woo}]{AvChWoWo13}
Avanzi, B., Cheung, E. C.~K., Wong, B., Woo, J.-K., 2013. On a periodic
  dividend barrier strategy in the dual model with continuous monitoring of
  solvency. Insurance: Mathematics and Economics 52~(1), 98--113.

\bibitem[{Avanzi et~al.(2007)Avanzi, Gerber, and Shiu}]{AvGeSh07}
Avanzi, B., Gerber, H.~U., Shiu, E. S.~W., 2007. Optimal dividends in the dual
  model. Insurance: Mathematics and Economics 41~(1), 111--123.

\bibitem[{Avanzi et~al.(2016)Avanzi, Tu, and Wong}]{AvTuWo16}
Avanzi, B., Tu, V.~W., Wong, B., 2016. On the interface between optimal
  periodic and continuous dividend strategies in the presence of transaction
  costs. {ASTIN} Bulletin 46~(3), 709--746.

\bibitem[{Avram et~al.(2017)Avram, Grahovac, and Vardar-Acarceren}]{AvGrVa17}
Avram, F., Grahovac, D., Vardar-Acarceren, C., 2017. The {$W, Z$} scale
  functions kit for first passage problems of spectrally negative l\'evy
  processes, and applications to the optimization of dividends. Tech. rep.,
  arXiv preprint arXiv:1706.06841.

\bibitem[{Avram et~al.(2007)Avram, Palmowski, and Pistorius}]{AvPaPi07}
Avram, F., Palmowski, Z., Pistorius, M.~R., 2007. On the optimal dividend
  problem for a spectrally negative {L}{\'e}vy process. Annals of Applied
  Probability 17~(1), 156--180.

\bibitem[{Bayraktar and Egami(2008)}]{BaEg08}
Bayraktar, E., Egami, M., 2008. Optimizing venture capital investments in a
  jump diffusion model. Mathematical Methods of Operations Research 67~(1),
  21--42.

\bibitem[{Bayraktar et~al.(2013)Bayraktar, Kyprianou, and Yamazaki}]{BaKyYa12}
Bayraktar, E., Kyprianou, A.~E., Yamazaki, K., 2013. On optimal dividends in
  the dual model. {ASTIN} Bulletin 43~(3), 359--372.

\bibitem[{Bayraktar et~al.(2014)Bayraktar, Kyprianou, and Yamazaki}]{BaKyYa13}
Bayraktar, E., Kyprianou, A.~E., Yamazaki, K., 2014. Optimal dividends in the
  dual model under transaction costs. Insurance: Mathematics and Economics 54,
  133--143.

\bibitem[{Bertoin(1998)}]{Ber98}
Bertoin, J., 1998. L{\'e}vy Processes. Cambridge Tracts in Mathematics.
  Cambridge University Press, Cambridge, UK.

\bibitem[{Borch(1967)}]{Bor67}
Borch, K., 1967. The theory of risk. Journal of the Royal Statistical Society.
  Series B (Methodological) 29~(3), 432--467.

\bibitem[{B{\"u}hlmann(1970)}]{Buh70}
B{\"u}hlmann, H., 1970. Mathematical Methods in Risk Theory. Grundlehren der
  mathematischen Wissenschaften. Springer-Verlag, Berlin, Heidelberg, New York.

\bibitem[{Chaumont and Doney(2005)}]{ChDo05}
Chaumont, L., Doney, R., 2005. On {L\'evy} processes conditioned to stay
  positive. Electronic Journal of Probability 10~(28), 948--961.

\bibitem[{Chen et~al.(2017)Chen, Yang, and Yongxia}]{ChYaZh17}
Chen, P., Yang, H., Yongxia, Z., 2017. Optimal periodic dividend and capital
  injection problem for spectrally positive l\'evy processes. Insurance:
  Mathematics and Economics 74, 135--146.

\bibitem[{Cheung and Wong(2017)}]{ChWo17}
Cheung, E., Wong, J., 2017. On the dual risk model with parisian implementation
  delays in dividend payments. European Journal of Operational Research 257,
  159--173.

\bibitem[{Cheung et~al.(2017)Cheung, Yang, and Zhang}]{ChYaZh17P}
Cheung, E., Yang, H., Zhang, Z., 2017. L\'evy insurance risk process with
  poissonian taxation. Scandinavian Actuarial Journal 2017~(1), 51--87.

\bibitem[{Cram{\'e}r(1930)}]{Cra30}
Cram{\'e}r, H., 1930. On the mathematical theory of risk. Skand. Jubilee
  Volume. Stockholm.

\bibitem[{de~Finetti(1957)}]{deF57}
de~Finetti, B., 1957. Su un'impostazione alternativa della teoria collettiva
  del rischio. Transactions of the XVth International Congress of Actuaries 2,
  433--443.

\bibitem[{Furrer(1998)}]{Fur98}
Furrer, H., 1998. Risk processes perturbed by alpha-stable {L\'evy} motion.
  Scandinavian Actuarial Journal 1998~(1), 59--74.

\bibitem[{Jeanblanc-Picqu{\'e} and Shiryaev(1995)}]{JeSh95}
Jeanblanc-Picqu{\'e}, M., Shiryaev, A.~N., 1995. Optimization of the flow of
  dividends. Russian Mathematical Surveys 50~(2), 257--277.

\bibitem[{Kuznetsov et~al.(2013)Kuznetsov, Kyprianou, and Rivero}]{KuKyRi13}
Kuznetsov, A., Kyprianou, A.~E., Rivero, V., 2013. The theory of scale
  functions for spectrally negative {L}\'evy processes. In: {L}\'evy Matters
  II, Springer Lecture Notes in Mathematics. Springer.

\bibitem[{Kyprianou(2006)}]{Kyp06}
Kyprianou, A.~E., 2006. Introductory lectures on fluctuations of {L}\'evy
  processes with applications. Springer, Berlin.

\bibitem[{Kyprianou(2014)}]{Kyp14}
Kyprianou, A.~E., 2014. Introductory Lectures on Fluctuations of L\'evy
  Processes with Applications, 2nd Edition. Springer-Verlag: Berlin.

\bibitem[{Loeffen(2008{\natexlab{a}})}]{Loe08a}
Loeffen, R., 2008{\natexlab{a}}. An optimal dividends problem with transaction
  costs for spectrally negative {L\'evy} processes. Radon Institute for
  Computational and Applied Mathematics, Austrian Academy of Sciences.

\bibitem[{Loeffen(2009)}]{Loe09}
Loeffen, R., 2009. An optimal dividends problem with a terminal value for
  spectrally negative l{\'e}vy processes with a completely monotone jump
  density. Journal of Applied Probability 46~(1), 85--98.

\bibitem[{Loeffen et~al.(2014)Loeffen, Renaud, and Zhou}]{LoReZh14}
Loeffen, R., Renaud, J., Zhou, X., 2014. Occupation times of intervals until
  first passage times for spectrally negative {L\'evy} processes. Stochastic
  Processes and their Applications 124~(3), 1408--1435.

\bibitem[{Loeffen(2008{\natexlab{b}})}]{Loe08}
Loeffen, R.~L., 2008{\natexlab{b}}. On optimality of the barrier strategy in de
  {F}inetti's dividend problem for spectrally negative {L}\'evy processes.
  Annals of Applied Probability 18~(5), 1669--1680.

\bibitem[{Lundberg(1909)}]{Lun09}
Lundberg, F., 1909. {\"U}ber die {T}heorie der {R}{\"u}ckversicherung.
  Transactions of the VIth International Congress of Actuaries 1, 877--948.

\bibitem[{Mazza and Rulli{\`e}re(2004)}]{MaRu04}
Mazza, C., Rulli{\`e}re, D., 2004. A link between wave governed random motions
  and ruin processes. Insurance: Mathematics and Economics 35~(2), 205--222.

\bibitem[{Noba et~al.(2018)Noba, P\'erez, Yamazaki, and Yano}]{NoPeYaYa17}
Noba, K., P\'erez, J.-L., Yamazaki, K., Yano, K., 2018. On optimal periodic
  dividend strategies for {L\'evy} processes. Insurance: Mathematics and
  Economics 80, 29--44.

\bibitem[{Pardo et~al.(2015)Pardo, P\'erez, and Rivero}]{PaPeRi15}
Pardo, J., P\'erez, J., Rivero, V., 2015. The excursion measure away from zero
  for spectrally negative {L\'evy} processes. Annales de l'Institut Henri
  Poincar\'e Probabilit\'es et Statistiques 54~(1).

\bibitem[{P\'erez and Yamazaki(2018)}]{PeYa16b}
P\'erez, J., Yamazaki, K., 2018. Mixed periodic-classical barrier strategies
  for {L}\'evy risk processes. Risks 6~(2).

\bibitem[{P\'erez and Yamazaki(2017)}]{PeYa16}
P\'erez, J.-L., Yamazaki, K., 2017. On the optimality of periodic barrier
  strategies for a spectrally positive {L}\'evy process. Insurance: Mathematics
  and Economics 77, 1--13.

\bibitem[{Protter(2005)}]{Pro05}
Protter, P., 2005. Stochastic Integration and Differential Equations, 2nd
  Edition. Springer-Verlag, Berlin-Heidelberg.

\bibitem[{Tu(2017)}]{Tu17}
Tu, V., 2017. On optimal period dividend strategies in actuarial surplus
  models. Ph.D. thesis, UNSW.

\bibitem[{Yang et~al.(2013)Yang, Yang, and Zhang}]{YaYaZh13}
Yang, H., Yang, H., Zhang, Z., 2013. On a {Sparre} {Andersen} risk model
  perturbed by a spectrally negative {L\'evy} process. Scandinavian Actuarial
  Journal 2013~(3), 213--239.

\bibitem[{Yao et~al.(2011)Yao, Yang, and Wang}]{YaYaWa11}
Yao, D., Yang, H., Wang, R., 2011. Optimal dividend and capital injection
  problem in the dual model with proportional and fixed transaction costs.
  European Journal of Operational Research 211, 568--576.

\end{thebibliography}

\appendix

\section{Proof of Equation \eqref{eqt.ZqrZq.increasing}}\label{Appendix.Formula.ZqrZq}

\corrs{From equation (8.9) in \citet*{Kyp14}, we have for $y\geq 0$
$$\Zq(y)-\frac{\delta}{\phi_\delta}\Wq(y)={\mathbb{E}}(e^{-\delta {\tau}_{Y,y,0}^- };{\tau}_{Y,y,0}^-<\infty)\geq 0$$
and therefore by taking derivative of the function  $y\mapsto\Zq(y)e^{-\phi_\delta y}$, we can conclude that such function is decreasing. Hence we have}
\begin{equation*}
\Zq(x)e^{-\phi_\delta x}-\Zq(y)e^{-\phi_\delta y}\geq 0
\end{equation*}
for $0\leq x\leq y$. Further noting that for $0\leq x\leq y$ and $u\geq 0$ the function $u\mapsto \Wq(x+u)/\Wq(y+u)$ is an increasing function in $u$ and using \corrs{$\lim_{x\rightarrow\infty}e^{-\phi_\delta}\Wq(x)=1/{\psi_Y^\prime(\phi_\delta)}$ \citep*[see Lemma 3.3 in][]{KuKyRi13}} we get
\begin{align*}
&\Zq(x)-\Zq(y)\frac{\Wq(x+u)}{\Wq(y+u)}\\
> ~& \Zq(x)-\Zq(y)\sup_{u\geq 0}\frac{\Wq(x+u)}{\Wq(y+u)}\\
=~&\Zq(x)-\Zq(y)e^{-\phi_\delta(y-x)}\\
\geq ~& 0.
\end{align*}
This shows that for $u\geq 0$
$$\frac{\Wq(x+u)}{\Zq(x)}< \frac{\Wq(y+u)}{\Zq(y)},~0\leq x< y,$$
i.e. the function $s\mapsto \Wq(s+u)/\Zq(s)$ is an increasing fucntion. Consequently, by using the second expression in \eqref{Zqr.2} we have that 
$$
\frac{\Zqr(x)}{\Zq(x)}=\gamma\int_0^\infty e^{-\phiqr u}\frac{\Wq(x+u)}{\Zq(x)}du
$$
is an increasing function in $x$. This gives
$$
\frac{\partial}{\partial x}\frac{\Zqr(x)}{\Zq(x)}> 0
$$
as desired.

\section{Proof of Lemma \ref{verification.lemma}}\label{Appendix.A}

We express the L\'evy process $X$ using L\'evy-It\^o decomposition as 
\begin{equation}\label{X.Levy.Decomposition}
	X(t)=\corr{-ct}+\sigma B(t)+\int_{0+}^{t}\int_{|z|\geq 1}z\Prm +\limep \int_{0+}^{t}\int_{\epsilon<|z|< 1}z\big(\Prm-\Pi(dz)ds\big),
\end{equation}
where $\{B(t);t\geq 0\}$ is a standard Brownian motion and $\mathcal{N}$ is a Poisson random measure \corr{(independent of the Brownian motion $B=\{B(t);t\geq 0\}$)} in the measure space $(\corr{[0,\infty)\times[0,\infty)},\mathcal{B}[0,\infty)\times\mathcal{B}[0,\infty), ds\times\Pi(dz))$ (see e.g. \cite{Kyp14} Chapter 2 for details). 

For any $\pi\in\Pi_\kappa$, the corresponding surplus process $X^\pi(t)$ is a semi-martingale which takes the form 
\begin{equation}
	X^\pi(t)=X(t)-D^\pi(t).
\end{equation}
Here we note that $D^\pi$ is an adapted pure jump process which does not jump at the same time as $X$ a.s. \corr{In addition, Condition 2 implies that $H'$ is bounded on sets $[1/n,n]$ for all $n\in\mathbb{N}$.}

Let $(T_n)_{n\in\mathbb{N}}$ be the sequence of stopping times defined by $T_n:=\inf\{t>0:~X^\pi(t)>n~\mbox{or}~X^\pi(t)<1/n\}$, by applying the change of variables formula (Theorem II.32 of \citet*{Pro05}) to the stopped process $\{e^{-\delta(t\wedge T_n)}H(X^\pi(t\wedge T_n));t\geq 0\}$, conditioning on $X(0)=x$, we have
\begin{align}
	&e^{-\delta(t\wedge T_n)}H(X^\pi(t\wedge T_n))-H(x)\nonumber\\
	=~&\intte (-\delta)H(\xps)ds+\intte H'(\xps)dX^\pi(s)+\frac{1}{2}\intte H''(\xps)d[X^\pi,X^\pi]^c(s)\nonumber\\
	&+\sumte [H(X^\pi(s))-H(\xps)-H'(\xps)\Delta X^\pi(s)]\nonumber\\
	=~&\intte (-\delta)H(\xps)ds+\intte H'(\xps)dX(s)+\intte \frac{\sigma^2}{2}H''(\xps) ds\nonumber\\
	&+\intte [H(\xps+z)-H(\xps)-H'(\xps)z]\Prm\nonumber\\
	&+\sumte [H(X^\pi(s))-H(\xps)]1_{\{\Delta D^\pi(s)>0\}}.
\end{align}
Plugging in the formula of $X$ from (\ref{X.Levy.Decomposition}) \corr{and collecting the terms for $(\mathcal{L}-\delta)$}, we get
%
\begin{align}
	&e^{-\delta(t\wedge T_n)}H(X^\pi(t\wedge T_n))-H(x)\nonumber\\
	=~&\intte (\mathcal{L}-\delta)H(\xps)ds+\sumte [H(X^\pi(s))-H(\xps)]1_{\{\Delta D^\pi(s)>0\}}+M_X(t\wedge T_n),\label{ver.eqt1}
\end{align}
where 
\begin{align}
	M_X(t\wedge T_n)=~&\intte \sigma H'(\xps) dB(s)\nonumber\\
	&+\limep\int_{0+}^{t\wedge T_n}\int_{\epsilon<|z|< 1}e^{-\delta s}H'(\xps)z\big(\Prm-\Pi(z)ds\big)\nonumber\\
	&+\intte [H(\xps+z)-H(\xps)-H'(\xps)z1_{\{|z|< 1\}}]\big(\Prm-\Pi(dz)ds)
\end{align}

We further express (\ref{ver.eqt1}) as 
\begin{align}
	&e^{-\delta(t\wedge T_n)}H(X^\pi(t\wedge T_n))-H(x)\nonumber\\
	&+\intte [(\Delta D^\pi(s)-\kappa)1_{\{\Delta D^\pi(s)>0\}}+H(\xps -\Delta D^\pi(s))-H(\xps)]d N_\gamma(s)\nonumber\\
	=~&\intte \Big((\mathcal{L}-\delta)H(\xps)+\gamma [(\Delta D^\pi(s)-\kappa)1_{\{\Delta D^\pi(s)>0\}}+H(\xps -\Delta D^\pi(s))-H(\xps)] \Big)ds\nonumber\\
	&+M_X(t\wedge T_n)-\intte (\Delta D^\pi(s)-\kappa)1_{\{\Delta D^\pi(s)>0\}}d N_\gamma(s)\nonumber\\
	&+\intte [(\Delta D^\pi(s)-\kappa)1_{\{\Delta D^\pi(s)>0\}}+H(\xps -\Delta D^\pi(s))-H(\xps)]d\big( N_\gamma(s)-\gamma ds\big).\label{ver.eqt2}
\end{align}
By denoting $M(t\wedge T_n)=M_X(t\wedge T_n)+M_\gamma(t\wedge T_n)$, where
\begin{equation}
	M_\gamma(t\wedge T_n)=\intte [(\Delta D^\pi(s)-\kappa)1_{\{\Delta D^\pi(s)>0\}}+H(\xps -\Delta D^\pi(s))-H(\xps)]d\big( N_\gamma(s)-\gamma ds\big),
\end{equation}
we can rewrite (\ref{ver.eqt2}) as 
\begin{align}
	&e^{-\delta(t\wedge T_n)}H(X^\pi(t\wedge T_n))-H(x)\nonumber\\
	=~&\intte \Big((\mathcal{L}-\delta)H(\xps)+\gamma [(\Delta D^\pi(s)-\kappa)1_{\{\Delta D^\pi(s)>0\}}+H(\xps -\Delta D^\pi(s))-H(\xps)] \Big)ds\nonumber\\
	&+M(t\wedge T_n)-\intte (\Delta D^\pi(s)-\kappa)1_{\{\Delta D^\pi(s)>0\}}d N_\gamma(s)\nonumber
\end{align}
where by condition 3 we have 
\begin{equation*}
	e^{-\delta(t\wedge T_n)}H(X^\pi(t\wedge T_n))-H(x)\leq M(t\wedge T_n)-\intte (\Delta D^\pi(s)-\kappa)1_{\{\Delta D^\pi(s)>0\}}d N_\gamma(s)
\end{equation*}
or equivalently
\begin{align*}
	H(x)\geq~& \intte (\Delta D^\pi(s)-\kappa)1_{\{\Delta D^\pi(s)>0\}}d N_\gamma(s)+ e^{-\delta(t\wedge T_n)}H(X^\pi(t\wedge T_n))-M(t\wedge T_n)\\
	\geq~&\intte (\Delta D^\pi(s)-\kappa)1_{\{\Delta D^\pi(s)>0\}}d N_\gamma(s)-M(t\wedge T_n)
\end{align*}
since $e^{-\delta(t\wedge T_n)}H(X^\pi(t\wedge T_n))\geq 0$ by Condition 2.

Condition 1 implies that $M(t\wedge T_n)$ is a zero mean martingale, hence by taking expectation we have
\begin{equation*}
	H(x)\geq \mathbb{E}_x[\intte (\Delta D^\pi(s)-\kappa)1_{\{\Delta D^\pi(s)>0\}}d N_\gamma(s)].
\end{equation*}
Finally, note that $T_n\rightarrow \tau^\pi$ a.s. and that by Condition 2 $H\geq0$. By applying Fatou's lemma, we have
\begin{align}
	H(x)\geq~& \limtn\mathbb{E}_x[\intte (\Delta D^\pi(s)-\kappa)1_{\{\Delta D^\pi(s)>0\}}d N_\gamma(s)]\nonumber\\
	\geq~&\mathbb{E}_x[\int_{0+}^{\tau^\pi}e^{-\delta s} (\Delta D^\pi(s)-\kappa)1_{\{\Delta D^\pi(s)>0\}}d N_\gamma(s)]\nonumber\\
	=~&V_\kappa(x;\pi).
\end{align}

\section{Proof of Theorem \ref{lemma.value.fcn}}\label{Appendix.B}

We proceed using exiting identities \citep*[from][]{AlIvZh16} together with strong Markov properties, which is a standard probabilistic argument.
In particular, we will borrow some results from \citet*{ChYaZh17}, where the expected values of interest are computed. 

Due to the nature of the strategy, the surplus process is controlled only when $X^{\pi_{b_u,b_l}}\geq b_u$. Therefore, we derive the expressions of $V_\kappa(x;\pi_{b_u,b_l})$ for $x\geq b_u$ and $x< b_u$ separately. We start with the case when $x\geq b_u$. We now define the following quantities:
\begin{align}
\tau^+_b &=\inf\{t\geq 0:X(t)>b\}\\
\tau^-_a &=\inf\{t\geq 0:X(t)<a\}\\
T^+_b&=\min\{T_i:X(T_i)>b\}\\
e_q &\sim ~\mbox{Exponential \corr{random variable} with mean}~1/q.
\end{align}

The value function for $x\geq b_u$ is given by the following lemma.
\begin{lemma}\label{lemma.2.2}
	For $x\geq b_u$, we have 
	\begin{equation}
	V_\kappa(x;\pi_{b_u,b_l}) = \frac{\gamma}{\gamma+\delta} \Auqr(x-b_u;d)+e^{-\phiqr (x-b_u)} V_\kappa(b_u;\pi_{b_u,b_l}) + \frac{\gamma}{\gamma+\delta} (1-e^{-\phiqr (x-b_u)}) V_\kappa(b_l;\pi_{b_u,b_l}). \label{Value.upper.non.opt}
	\end{equation}
\end{lemma}
\begin{proof}[Proof of Lemma \ref{lemma.2.2}]
Similar to Lemma 3.4 and Theorem 3.1 in \citet*{ChYaZh17}, we have for $x\geq b_u$,
	\begin{align}
	&\mathbb{E}_x\left[e^{-\delta T_1} (X_{T_1}-b_l-\kappa) 1_{\{T_1<\tau_{b_u}^-\}} \right] 
	= \frac{\gamma}{\gamma+\delta}\left((x-b_u)+(b_u-b_l-\kappa+\frac{\mu}{\gamma+\delta})(1-e^{-\phiqr (x-b_u)}) \right),\label{lemma2.1eq1} \\
	&\mathbb{E}_x\left[e^{-\delta \tau_{b_u}^-} 1_{\{\tau_{b_u}^-<T_1\}} \right] 
	= e^{-\phiqr (x-b_u)}, \label{lemma2.1eq3}\\
	&\mathbb{E}_x\left[e^{-\delta T_1} 1_{\{T_1<\tau_{b_u}^-\}} \right] 
	=\frac{\gamma}{\gamma+\delta}\left( 1-e^{-\phiqr (x-b_u)}\right),\label{lemma2.1eq2} 
	\end{align}
	and \corr{by} the strong Markov property 
	\begin{align}
		V_\kappa(x;\pi_{b_u,b_l}) =~&  \mathbb{E}_x\Big[e^{-\delta T_1} (X_{T_1}-b_l-\kappa+V_\kappa(b_l;\pi_{b_u,b_l})) 1_{\{T_1<\tau_{b_u}^-\}} \Big] +\mathbb{E}_x\Big[e^{-\delta \tau_{b_u}^-} 1_{\{\tau_{b_u}^-<T_1\}} \Big] V_\kappa(b_u;\pi_{b_u,b_l})\\
		=~&  \mathbb{E}_x\Big[e^{-\delta T_1} (X_{T_1}-b_l-\kappa) 1_{\{T_1<\tau_{b_u}^-\}} \Big] +\mathbb{E}_x\Big[e^{-\delta \tau_{b_u}^-} 1_{\{\tau_{b_u}^-<T_1\}} \Big] V_\kappa(b_u;\pi_{b_u,b_l})\nonumber\\
		&+ \mathbb{E}_x\Big[e^{-\delta T_1} 1_{\{T_1<\tau_{b_u}^-\}} \Big]  V_\kappa(b_l;\pi_{b_u,b_l}),\label{lemma2.eq.4}
	\end{align}
	where substituting the expected values in (\ref{lemma2.eq.4}) using (\ref{lemma2.1eq1})-(\ref{lemma2.1eq2}) gives (\ref{Value.upper.non.opt}).
\end{proof}

Next, we consider the case when $x<b_u$.
\begin{lemma}\label{lemma2.3}
	For $x<b_u$, we have 
	\begin{align}
	&\mathbb{E}_x\left[ e^{-\delta T_{b_u}^+} (X_{T_{b_u}^+}-b_u) 1_{\{T_{b_u}^+<\tau_0^-\}} \right] = \Bqr(b_u-x;b_u),\\
	&\mathbb{E}_x\left[ e^{-\delta T_{b_u}^+} 1_{\{T_{b_u}^+<\tau_0^-\}} \right] = \Lqr(b_u-x;b_u),
	\end{align}
	where
	\begin{align}
	&\Kqr(x) = -\frac{\gamma}{\gamma+\delta}\frac{\mu}{\delta}\Zq(x) - {\overline{Z}}_\delta(x)+\frac{\mu}{\delta},\\
	&\Bqr(x;b_u) = \frac{\gamma}{\gamma+\delta}\left[ \Kqr(x)-\frac{\Zqr(x)}{\Zqr(b_u)}\Kqr(b_u)\right],\\
	&\Lqr(x;b_u) =  \frac{\gamma}{\gamma+\delta}\left[ \Zq(x)-\frac{\Zqr(x)}{\Zqr(b_u)}\Zq(b_u)\right].
	\end{align}
\end{lemma}
\begin{proof}[Proof of Lemma \ref{lemma2.3}]
	See Lemma 3.5 in \citet*{ChYaZh17}.
\end{proof}

For the functions $\Bqr(\cdot;b_u)$ and $L_{\gamma,\delta}(\cdot;b_u)$ defined, we present the following identities, which will be used in the next proof:
\begin{align}
&\Bqr(x;b_u) = -\frac{\gamma}{\gamma+\delta}\frac{\mu}{\delta}\Jqr(x)-\Hqr(x)+ \Zqr(x) \Bqr(0;b_u),\label{identityL1}\\
&\Lqr(x;b_u) = \frac{\gamma}{\gamma+\delta}\Zq(x) - \Zqr(x)\left( \frac{\gamma}{\gamma+\delta} - \Lqr(0;b_u)\right).\label{identityL2}
\end{align}

The value function for $x<b_u$ is given by the following lemma.
\begin{lemma}\label{lemma.2.4}
	For $x<b_u$, we have
	\begin{equation}\label{value.eq.low.general}
	V_\kappa(x;\pi_{b_u,b_l}) = \frac{\gamma}{\gamma+\delta}\Aqr(b_u-x;d)+\Zqr(b_u-x) V_\kappa(b_u;\pi_{b_u,b_l}) + \frac{\gamma}{\gamma+\delta} (\Zq(b_u-x)-\Zqr(b_u-x)) V_\kappa(b_l;\pi_{b_u,b_l}).
	\end{equation}
\end{lemma}
\begin{proof}[Proof of Lemma \ref{lemma.2.4}]
	We first note by strong Markov property and Lemma \ref{lemma2.3} that
	\begin{equation}\label{valuefcnloweq1}
	V_\kappa(x;\pi_{b_u,b_l}) = \Bqr(b_u-x;b_u)+\left(b_u-b_l-\kappa+V_\kappa(b_l;\pi_{b_u,b_l})\right)\Lqr(b_u-x;b_u).
	\end{equation}
	\corr{Letting $x$ go to $b_u$}, we have 
	\begin{equation*}
	V_\kappa(b_u;\pi_{b_u,b_l}) = \Bqr(0;b_u)+\left(b_u-b_l-\kappa+V_\kappa(b_l;\pi_{b_u,b_l})\right)\Lqr(0;b_u),
	\end{equation*}
	or
	\begin{equation}\label{eqtL0}
	\Bqr(0;b_u) = V_\kappa(b_u;\pi_{b_u,b_l}) - \Lqr(0;b_u)\left( b_u-b_l-\kappa+V_\kappa(b_l;\pi_{b_u,b_l})\right) .
	\end{equation}
	Inserting (\ref{identityL1}) and (\ref{identityL2}) into (\ref{valuefcnloweq1}) and further simplifying using (\ref{eqtL0}), we obtain the result.
\end{proof}

It should be clear that the value function $V_\kappa(\cdot;\pi_{b_u,b_l})$ can be expressed in terms of $V_\kappa(b_u;\pi_{b_u,b_l})$ and $V_\kappa(b_l;\pi_{b_u,b_l})$ for all $x\geq 0$. We still need to find the 2 constants. We divide it into two cases depending on whether $b_l=0$ \corr{or not}.

When $b_l=0$, a liquidation at first opportunity strategy,  in view of (\ref{value.eq.low.general}), using $V_\kappa(b_l;\pi_{b_u,b_l})=V_\kappa(0;\pi_{b_u,b_l})=0$, we have
\begin{equation*}
V_\kappa(x;\pi_{b_u,b_l})=\frac{\gamma}{\gamma+\delta}\Aqr(b_u-x;d)+\Zqr(b_u-x)V_\kappa(b_u;\pi_{b_u,b_l}).
\end{equation*}
Further substituting $x=0$ and using $V_\kappa(0;\pi_{b_u,b_l})=0$, we obtain
\begin{equation*}
0=V_\kappa(0;\pi_{b_u,b_l})=\frac{\gamma}{\gamma+\delta}\Aqr(b_u;d)+V_\kappa(b_u;\pi_{b_u,b_l})\corr{\Zqr(b_u)},
\end{equation*}
or equivalently
\begin{equation}
V_\kappa(b_u;\pi_{b_u,b_l})=\frac{\gamma}{\gamma+\delta}\frac{-\Aqr(b_u;d)}{\Zqr(b_u)}.
\end{equation}

When $b_l>0$, substituting $x=b_l$ and $x=0$ in (\ref{value.eq.low.general}) and noticing $V_\kappa(0;\pi_{b_u,b_l})=0$, we have
\begin{align}
\Zqr(b_u)V_\kappa(b_u;\pi_{b_u,b_l})+\frac{\gamma}{\gamma+\delta}(\Zq(b_u)-\Zqr(b_u))V_\kappa(b_l;\pi_{b_u,b_l})&=-\frac{\gamma}{\gamma+\delta}\Aqr(b_u;d)\label{solve.vbu.vbl.eq1}\\
\Zqr(d)V_\kappa(b_u;\pi_{b_u,b_l})+\left( \frac{\gamma}{\gamma+\delta}\left(\Zq(d)-\Zqr(d)\right)-1\right) V_\kappa(b_l;\pi_{b_u,b_l})&=-\frac{\gamma}{\gamma+\delta}\Aqr(d;d).\label{solve.vbu.vbl.eq2}
\end{align}
To solve $V_\kappa(b_u;\pi_{b_u,b_l})$ and $V_\kappa(b_l;\pi_{b_u,b_l})$, we need to make sure that the determinant of \\$ \left( \begin{array}{ccc}
\Zqr(b_u) & \frac{\gamma}{\gamma+\delta}(\Zq(b_u)-\Zqr(b_u))  \\
\Zqr(d) & \frac{\gamma}{\gamma+\delta}(\Zq(b_u)-\Zqr(b_u))-1  \end{array} \right)$
is non-zero. This property can be checked by noticing that it is always negative, i.e.
\begin{align*}
&\Zqr(b_u)\left( \frac{\gamma}{\gamma+\delta}(\Zq(d)-\Zqr(d))-1\right)-\Zqr(d)\frac{\gamma}{\gamma+\delta}(\Zq(b_u)-\Zqr(b_u))<0\\
\iff&\frac{\gamma}{\gamma+\delta}\left( \Zq(d)\corr{\Zqr(b_u)}-\Zqr(d)\Zq(b_u)\right) <\Zqr(b_u)\\
\iff&\frac{\gamma}{\gamma+\delta}\left( \Zq(d)-\frac{\Zqr(d)}{\Zqr(b_u)}\Zq(b_u)\right) <1,
\end{align*}
which is always true since the expression on the left hand side is equal to $	\mathbb{E}_d\left[ e^{-\delta T^-_{Y,0}}1_{\{T^-_{Y,0}<\tau_{Y,b_u}^+\}}\right]<1$ by (15) in \cite*{AlIvZh16}. Hence, we are able to solve $V_\kappa(b_u;\pi_{b_u,b_l})$ and $V_\kappa(b_l;\pi_{b_u,b_l})$ using (\ref{solve.vbu.vbl.eq1}) and (\ref{solve.vbu.vbl.eq2}).
The values of the constants, $V_\kappa(b_u;\pi_{b_u,b_l})$ and $V_\kappa(b_l;\pi_{b_u,b_l})$, are given by 
\begin{align}
V_\kappa(b_l;\pi_{b_u,b_l})=~&\frac{(d-\kappa-\frac{\gamma}{\gamma+\delta}\frac{\mu}{\delta})\corr{\Zqr(b_u)}\Lqr(d;b_u)+\Hqr(b_u)\Zqr(d)-\Hqr(d)\Zqr(b_u)}{\Zqr(b_u)(1-\Lqr(d;b_u))}\\
V_\kappa(b_u;\pi_{b_u,b_l})=~&\frac{\frac{-\gamma}{\gamma+\delta}\Aqr(b_u;d)-\frac{\gamma}{\gamma+\delta}\frac{\mu}{\delta}\corr{\Zqr(b_u)}\Lqr(d;b_u)+\Cqr(b_u)\Hqr(d)-\Cqr(d)\Hqr(b_u)}{\Zqr(b_u)(1-\Lqr(d;b_u))}.
\end{align}

\section{Proof of Lemma \ref{lemma.smoothness.equivalent} and Proposition \ref{lemma.existence.bu}}\label{Appendix.C}

\subsection{Proof of Lemma \ref{lemma.smoothness.equivalent}}\label{App.C1}

We first investigate the case when $b_l=0$. When $b_l=0$, recall from Remark \ref{R_proofsmooth} that \eqref{eq.to.be.solve.bl.0} holds, which is equivalent to
\begin{equation*}
(b_u-\kappa)\Zqr(b_u)+\frac{\gamma}{\gamma+\delta}\Aqr(b_u;d)=0.
\end{equation*}
Inserting the expressions of $\Aqr(\cdot;d)$ from (\ref{A.expression}), we have
\begin{align*}
0=~&\Zqr(b_u)(b_u-\kappa)+\frac{\gamma}{\gamma+\delta}\left( \frac{-\mu}{\delta}\Jqr(b_u)+\frac{\mu}{\delta}-{\overline{Z}}_\delta(b_u)+(b_u-\kappa)(\Zq(b_u)-\Zqr(b_u))\right) \\
=~&(b_u-\kappa)\Jqr(b_u)+\frac{\gamma}{\gamma+\delta}\left( \frac{\mu}{\delta}-\frac{\mu}{\delta}\Jqr(b_u)-{\overline{Z}}\de(b_u)\right),
\end{align*}
which is the same as $\Gamma_{b_l}$ defined in (\ref{bubleqt}) with $b_l=0$.

When $b_l>0$, equating  $V_{\kappa}(b_u;\pi_{b_u,b_l})-V_{\kappa}(b_l;\pi_{b_u,b_l})=b_u-b_l-\kappa$ with the help of (\ref{value.bl}) and (\ref{value.bu}), after some algebra, one can show that (\ref{smoothness.condition}) is equivalent to (\ref{bubleqt}), i.e.
\begin{equation*}
\Gamma_{b_l}(d)=\left(d-\kappa-\frac{\gamma}{\gamma+\delta}\frac{\mu}{\delta}\right) \Jqr(b_u)-\Hqr(b_u)+\Jqr(d)\Hqr(b_u)-\Jqr(b_u)\Hqr(d)=0,
\end{equation*}
where $b_u=b_l+d$.

\subsection{Proof of Proposition \ref{lemma.existence.bu}}\label{App.C2}
	The goal is to show that there is a unique root for $\Gamma_{b_l}(d)=0$, $b_l\in[0,b^*]$. 
	
	We first show the existence of a root.
	\corr{This is achieved by showing (1) $\frac{\partial}{\partial d}\Gamma_{b_l}(d)$ goes to infinity when $d$ goes to infinity and (2) $\Gamma_{b_l}(\kappa)<0$ so that a root for $\Gamma_{b_l}(d)=0$ exists by continuity.}
	
	To compute the derivative, we will use the following identity: 
	\begin{equation}
	\Hqr(b_u)\Zqr(d)-\Hqr(d)\Zqr(b_u)=\frac{\gamma}{\gamma+\delta}\frac{\mu}{\delta}\Lqr(d;b_u)\Zqr(b_u)+\Bqr(d;b_u)\Zqr(b_u).
	\end{equation}
	This can be shown by the following:
	\begin{align*}
	\Bqr(d;b_u)=~&\frac{\gamma}{\gamma+\delta}\left(K(d)-\frac{\Zqr(d)}{\Zqr(b_u)}K(b_u)\right)\\
	=~&\left(-\frac{\gamma}{\gamma+\delta}\frac{\mu}{\delta}\frac{\gamma}{\gamma+\delta}\Zq(d)-\Hqr(d)\right)-\frac{\Zqr(d)}{\Zqr(b_u)}\left(-\frac{\gamma}{\gamma+\delta}\frac{\mu}{\delta}\frac{\gamma}{\gamma+\delta}\Zq(b_u)-\Hqr(b_u)\right)\\
	=~&-\frac{\gamma}{\gamma+\delta}\frac{\mu}{\delta}\frac{\gamma}{\gamma+\delta}\left(\Zq(d)-\frac{\Zqr(d)}{\Zqr(b_u)}\Zq(b_u)\right)+\frac{1}{\Zqr(b_u)}\left(\Hqr(b_u)\Zqr(d)-\Hqr(d)\Zqr(b_u)\right)\\
	=~&-\frac{\gamma}{\gamma+\delta}\frac{\mu}{\delta}\Lqr(d;b_u)+\frac{1}{\Zqr(b_u)}\left(\Hqr(b_u)\Zqr(d)-\Hqr(d)\Zqr(b_u)\right).
	\end{align*}
	Hence, we have
	\begin{align}
	&\Jqr\p(d)\Hqr(b_u)-\Jqr\p(b_u)\Hqr(d)\nonumber\\
	=~&\frac{\delta}{\gamma+\delta}\phiqr\left(\Hqr(b_u)\Zqr(d)-\Hqr(d)\Zqr(b_u)\right)\nonumber\\
	=~&\frac{\delta}{\gamma+\delta} \Zqr(b_u)\phiqr\left(\frac{\gamma}{\gamma+\delta}\frac{\mu}{\delta}\Lqr(d;b_u)+\Bqr(d;b_u)\right).
	\end{align}
	Furthermore, by direct computation, we have
	\begin{equation}
	\Jqr(d)\Hqr\p(b_u)-\Jqr(b_u)\Hqr\p(d) = -\frac{\delta}{\gamma+\delta}\Lqr(d;b_u)\corr{\Zqr(b_u)}.
	\end{equation}
	Therefore, we have
	\begin{align}
	&\Gamma_{b_l}'(d)\nonumber\\
	=~&(d-\kappa-\frac{\gamma}{\gamma+\delta}\frac{\mu}{\delta})\Jqr\p(b_u)+\Jqr(b_u)-\frac{\gamma}{\gamma+\delta}\Zq(b_u)+(\Jqr\p(d)\Hqr(b_u)-\Jqr\p(b_u)\Hqr(d))\nonumber\\
	&+(\Jqr(d)\Hqr\p(b_u)-\Jqr(b_u)\Hqr\p(d))\nonumber\\
	=~&\frac{\delta}{\gamma+\delta}\Zqr(b_u)\left(\phiqr(d-\kappa-\frac{\gamma}{\gamma+\delta}\frac{\mu}{\delta})+1-\Lqr(d;b_u)+\phiqr(\frac{\gamma}{\gamma+\delta}\frac{\mu}{\delta}\Lqr(d;b_u)+\Bqr(d;b_u))\right)\nonumber\\
	=~&\frac{\delta}{\gamma+\delta}\Zqr(b_u)\left(\left(1-\Lqr(d;b_u)\right)+\phiqr\left(\Bqr(d;b_u)+d-\kappa-\frac{\gamma}{\gamma+\delta}\frac{\mu}{\delta}(1-\Lqr(d;b_u))\right)\right)\nonumber\\
	=~&\frac{\delta}{\gamma+\delta}\phiqr \Zqr(b_u)\left(1-\Lqr(d;b_u)\right)\left(\frac{\Bqr(d;b_u)}{1-\Lqr(d;b_u)}+\frac{d-\kappa}{1-\Lqr(d;b_u)}-\left(\frac{\gamma}{\gamma+\delta}\frac{\mu}{\delta}-\frac{1}{\phiqr}\right)\right).\label{Gamma.dev.bl}
	\end{align}
	Using (\ref{valuefcnloweq1}), putting $x=b_l$, we have
	\begin{align}
	&V_\kappa(b_l;\pi_{b_u,b_l})=(d-\kappa+V_\kappa(b_l;\pi_{b_u,b_l}))\Lqr(d;b_u)+\Bqr(d;b_u)\nonumber\\
	\iff&(d-\kappa)\corr{\Lqr(d;b_u)}+\Bqr(d;b_u)=V_\kappa(b_l;\pi_{b_u,b_l})(1-\Lqr(d;b_u))\nonumber\\
	\iff&\frac{\Bqr(d;b_u)}{1-\Lqr(d;b_u)}=V_\kappa(b_l;\pi_{b_u,b_l})-\frac{(d-\kappa)\Lqr(d;b_u)}{1-\Lqr(d;b_u)}\label{identity.vbl}.
	\end{align}
	Substituting (\ref{identity.vbl}) into (\ref{Gamma.dev.bl}), we get 
	\begin{align}
	&\Gamma_{b_l}'(d)\nonumber\\
	=~&\frac{\delta}{\gamma+\delta}\phiqr \Zqr(b_u)\left(1-\Lqr(d;b_u)\right)\left(\frac{\Bqr(d;b_u)}{1-\Lqr(d;b_u)}+\frac{d-\kappa}{1-\Lqr(d;b_u)}-\left(\frac{\gamma}{\gamma+\delta}\frac{\mu}{\delta}-\frac{1}{\phiqr}\right)\right)\nonumber\\
	=~&\frac{\delta}{\gamma+\delta}\phiqr \Zqr(b_u)\left(1-\Lqr(d;b_u)\right)\left(V_\kappa(b_l;\pi_{b_u,b_l})-\frac{(d-\kappa)\Lqr(d;b_u)}{1-\Lqr(d;b_u)}+\frac{d-\kappa}{1-\Lqr(d;b_u)}-\left(\frac{\gamma}{\gamma+\delta}\frac{\mu}{\delta}-\frac{1}{\phiqr}\right)\right)\nonumber\\
	=~&\frac{\delta}{\gamma+\delta}\phiqr \Zqr(b_u)\left(1-\Lqr(d;b_u)\right)\left(V_{\kappa}(b_l;\pi_{b_u,b_l})+d-\kappa-\left(\frac{\gamma}{\gamma+\delta}\frac{\mu}{\delta}-\frac{1}{\phiqr}\right)\right).\label{GammaBlD}
	\end{align}
	Hence, we have
	\begin{equation}
	\lim_{d\rightarrow\infty}\Gamma_{b_l}'(d)=+\infty
	\end{equation}
	since from \eqref{Zqr.2} $\Zqr(x)=\gamma\int_0^\infty e^{-\phiqr y}\Wq(x+y)dy$  is increasing in $x$, $\Lqr(d;b_u)$ is bounded above by $\Lqr(\kappa;b_u)$, $\lim_{d\rightarrow\infty}\Lqr(d;b_u)=0$ and $V_{\kappa}(b_l;\pi_{b_u,b_l})$ is bounded below by $0$. \corr{This implies $$\lim_{d\rightarrow\infty}\Gamma_{b_l}(d)=\infty.$$}
	
	Next, we show that $\Gamma_{b_l}(\kappa)<0$. \corrs{Suppose $b^*=0$ and therefore we have $b_l=0$. In view of the definition of $\Gamma_{b_l}$ in \eqref{bubleqt}, we have
	\begin{equation}
	\Gamma_0(\kappa)=-\frac{\gamma}{\gamma+\delta}\frac{\mu}{\delta}\Jqr(\kappa)-\Hqr(\kappa)=-\frac{\gamma}{\gamma+\delta}\Big(\overline{Z}_{\delta}(\kappa)-\frac{\mu}{\delta}+\frac{\mu}{\delta}\Jqr(\kappa)\Big).
	\end{equation}
	Hence, it suffices to show that the term inside the last bracket above is strictly positive for $\kappa>0$. To do so, we use the following inequality for $0\leq x\leq b$,
	\begin{equation*}
	\overline{V}_b(x):=\frac{\Zq(b-x)}{\Zq(b)}\Big(\overline{Z}_\delta (b)-\frac{\mu}{\delta}\Big)-\Big(\overline{Z}_\delta (b-x)-\frac{\mu}{\delta}\Big)\geq 0,
	\end{equation*}
	which is justified by the fact that the function $\overline{V}_b$ defined above is the value function of a (continuous) barrier strategy with barrier level $b\geq 0$ in \citet*{BaKyYa12}, which is by definition non-negative. In particular, inserting $x=b$ in the above yields
	\begin{equation}\label{ineq.vb}
	0\leq \overline{V}_b(b)=\frac{1}{\Zq(b)}\Big(\overline{Z}_\delta (b)-\frac{\mu}{\delta}+\frac{\mu}{\delta}\Zq(b)\Big).
	\end{equation}
	To this end, we notice that equation \eqref{eqt.ZqrZq.increasing} readily yields $\Zqr(x)/\Zq(x)>\Zqr(0)/\Zq(0)=1$, which further implies $\Jqr(x)\geq \Zq(x)$, for $x> 0$. Thus, we have from \eqref{ineq.vb} that
	$$\Gamma_0(\kappa)=-\frac{\gamma}{\gamma+\delta}\Big(\overline{Z}_{\delta}(\kappa)-\frac{\mu}{\delta}+\frac{\mu}{\delta}\Jqr(\kappa)\Big)<-\frac{\gamma}{\gamma+\delta}\Big(\overline{Z}_{\delta}(\kappa)-\frac{\mu}{\delta}+\frac{\mu}{\delta}\Zq(\kappa)\Big)\leq 0.$$
	On the other hand, if $b^*>0$, in} view of the definition of $\Gamma_{b_l}$ in \eqref{bubleqt}, when $d=\kappa$, we have
	\begin{equation}
	\widetilde{\Gamma}_{b_l}(\kappa):=\Gamma_{b_l}(\kappa)=-\frac{\gamma}{\gamma+\delta}\frac{\mu}{\delta} \Jqr(b_l+\kappa)-\Hqr(b_l+\kappa)+\Jqr(\kappa)\Hqr(b_l+\kappa)-\Jqr(b_l+\kappa)\Hqr(\kappa),
	\end{equation}
	which is essentially the same as $\Gamma_{b_l}$ if we replace $d-\kappa$ and $\kappa$ by $0$ and $d$ respectively. In this sense, if we differentiate the above w.r.t. $\kappa$, we will obtain the formula \eqref{GammaBlD} except we do not have the term $\Jqr(b_u)$, i.e.
	$$\frac{\partial}{\partial \kappa}\widetilde\Gamma_{b_l}(\kappa)=\frac{\delta}{\gamma+\delta}\phiqr \Zqr(b_l+\kappa)\left(1-\Lqr(\kappa;b_l+\kappa)\right)\left(V_{\kappa}(b_l;\pi_{b_l+\kappa,b_l})-\left(\frac{\gamma}{\gamma+\delta}\frac{\mu}{\delta}-\frac{1}{\phiqr}\right)\right)-\Jqr(b_l+\kappa).$$
	Now, by noting that the term inside the last bracket in the first term is 
	$$V_{\kappa}(b_l;\pi_{b_l+\kappa,b_l})-\left(\frac{\gamma}{\gamma+\delta}\frac{\mu}{\delta}-\frac{1}{\phiqr}\right)\leq V_{\kappa}(b_l;\pi_{b_l+\kappa,b_l})-v_0(b^*)\leq 0$$
	for $b_l\leq b^*$ due to \eqref{V0bstar}. Therefore, we can conclude that 
	$$\Gamma_{b_l}(\kappa)=\tilde\Gamma_{b_l}(\kappa)<\tilde\Gamma_{b_l}(0)=0.$$

	\corr{Combining with $\lim_{d\rightarrow\infty}\Gamma_{b_l}(d)=\infty$, we can conclude that} there is a root for $\Gamma_{b_l}(d)=0$ \corr{provided that $0\leq b_l\leq b^*$}.
	
	Finally, we show the uniqueness of the root. Suppose there is a root $d'$ satisfying $\Gamma_{b_l}(d')=0$, then we have $V_{\kappa}(b_l;\pi_{b_u',b_l})+d'-\kappa=V_{\kappa}(b_u;\pi_{b_u',b_l})$ by the definition of $\Gamma_{b_l}$, where  $b_u'=b_l+d'$. From (\ref{Vbu.smooth}), we have
	\begin{equation*}
	V_{\kappa}(b_u';\pi_{b_u',b_l})=\left(\frac{\Hqr(b_u')}{\Jqr(b_u')}+\frac{1}{\phiqr}\right)+\left(\frac{\gamma}{\gamma+\delta}\frac{\mu}{\delta}-\frac{1}{\phiqr}\right)
	\end{equation*}
	and subsequently 
	\begin{align}
	&\Gamma_{b_l}'(d')\nonumber\\
	=~&\frac{\delta}{\gamma+\delta}\phiqr \Zqr(b_u')\left(1-\Lqr(d';b_u')\right)\left(V_{\kappa}(b_l;\pi_{b_u',b_l})+d'-\kappa-\left(\frac{\gamma}{\gamma+\delta}\frac{\mu}{\delta}-\frac{1}{\phiqr}\right)\right)\nonumber\\
	=~&\frac{\delta}{\gamma+\delta}\phiqr \Zqr(b_u')\left(1-\Lqr(d';b_u')\right)\left(V_{\kappa}(b_u';\pi_{b_u',b_l})-\left(\frac{\gamma}{\gamma+\delta}\frac{\mu}{\delta}-\frac{1}{\phiqr}\right)\right)\nonumber\\
	=~&\frac{\delta}{\gamma+\delta}\phiqr \Zqr(b_u')\left(1-\Lqr(d';b_u')\right)\left(\frac{\Hqr(b_u')}{\Jqr(b_u')}+\frac{1}{\phiqr}\right).
	\end{align}
	
	Now, if there is a $d_1$ such that $\Gamma_{b_l}(d_1)=0$ and $\Gamma_{b_l}'(d_1)<0$, there must exist another root $d_2$ such that $d_2<d_1$ and $\Gamma_{b_l}'(d_2)>0$ since $\Gamma_{b_l}'(d_1-)<0$, $\Gamma_{b_l}(\kappa)<0$ and $\Gamma_{b_l}$ is continuously differentiable. However, this leads to a contradiction as $\frac{\Hqr(b)}{\Jqr(b)}+\frac{1}{\phiqr}$ is negative if and only if $b<b^*$ and $b^*>0$ (see Section \ref{section.prelim.SPLP}).
	
	Next, if there is a $d_1$ such that $\Gamma_{b_l}(d_1)=0$ and $\Gamma_{b_l}'(d_1)=0$, we have $b_l+d_1=b^*$ and hence by (\ref{eqt.smooth}) 
	\[
	V_{\kappa}(x;\pi_{b_u',b_l})=\frac{\Hqr(b^*)}{\Jqr(b^*)}\Jqr(b^*-x)-\Hqr(b^*-x),~x\leq b^*.
	\]
	However, the right hand side of the equation is $v_{0}(x)$ by (\ref{eqt.v0rho}) while the left hand side is at most $v_{\kappa}(x)$. This is a contradiction as $v_{\kappa}(x)<v_{0}(x)$ for all $x\geq 0$.

	Therefore, it is impossible to have $\Gamma_{b_l}(d)=0$ and $\Gamma_{b_l}'(d)\leq 0$ at the same time. Hence we have  $\Gamma_{b_l}'(d_1)>0$ whenever $\Gamma_{b_l}(d_1)=0$, which implies that $d_1$ is the only root such that $\Gamma_{b_l}(d_1)=0$ as $\Gamma_{b_l}$ is continuous.
	
	Thus, there is one and only one root for $\Gamma_{b_l}(d)=0$.

	Finally, to show $b_u>b^*$, we observe
	\begin{equation} 
	\Gamma_{b_l}(d)=0\implies \Gamma_{b_l}'(d)> 0\iff \frac{\Hqr(b_u)}{\Jqr(b_u)}+\frac{1}{\phiqr}> 0\iff b_u>b^*.
	\end{equation}
\begin{remark}\label{remark.d.cont.in.bl}
	We should note that the root $d$ is continuous in $b_l$, because (i) the root is unique with strictly positive derivative and (ii) the \corr{formula of} $\Gamma_{b_l}(d)$ defined in (\ref{bubleqt})\corr{, as a function of $(b_l,d)$, is continuous}.
\end{remark}

\section{Proof of Lemma \ref{lemma.kappa0.exist}}\label{Appendix.D}

We first establish the following:
\begin{enumerate}
	\item $b-\frac{\Hqr(b)}{\Jqr(b)}$ as a function of $b$ is strictly increasing on $[b^*,\infty)$, and
	\item $\frac{\partial}{\partial b}V_0(h;\corr{\pi_b})<0$ for $0<h<b$, and
	\item $\lim_{b\rightarrow \infty}{V}_0(h;\pi_b)=0$ for any $h\geq 0$.
\end{enumerate}

\begin{proof}[Proof of Property 1]
Via differentiating with respect to $b$, we get
	\begin{align}
\frac{\partial}{\partial b_u}(b_u-\frac{\Hqr{(b_u)}}{\Jqr(b_u)})=~&
1-\frac{\Hqr^\prime(b_u)}{\Jqr(b_u)}+\frac{\Hqr(b_u)\Jqr^\prime(b_u)}{\Jqr(b_u)^2}\nonumber\\
=~&\frac{\Jqr(b_u)-\frac{\gamma}{\gamma+\delta}\Zq(b_u)}{\Jqr(b_u)}+\frac{\delta}{\gamma+\delta}\frac{\Zqr(b_u)}{\Jqr(b_u)}\frac{\phiqr\Hqr(b_u)}{\Jqr(b_u)}\nonumber\\
=~&\frac{\delta}{\gamma+\delta}\frac{\Zqr(b_u)}{\Jqr(b_u)}\Big(\frac{\phiqr\Hqr(b_u)}{\Jqr(b_u)}+1\Big)\label{eqt.DHJ}\\
>~&0\nonumber
\end{align}
because of (\ref{eqt.smooth.dev.positive}). Hence, the proof is complete.
\end{proof}
\begin{proof}[Proof of Property 2]
Note that $\log \Jqr$ is strictly increasing because $\Jqr$ is strictly increasing, which gives
\begin{equation}
\frac{\partial}{\partial b}\log \Jqr(b)=\frac{\delta}{\gamma+\delta}\phiqr\frac{\Zqr(b)}{\Jqr(b)}>0.\label{eqt.dev.ZqroverJ.geq0}
\end{equation}
This together with (\ref{eqt.DHJ}) imply that for all $0<h<b$ and $b>b^*$
\begin{align*}
&\frac{\partial}{\partial b}V_0(h;\pi_b)\\=~&	\frac{\partial}{\partial b}(\frac{\Hqr(b)}{\Jqr(b)}\Jqr(b-u)-\Hqr(b-u))\\
=~&\frac{\partial}{\partial b}(\frac{\Hqr(b)}{\Jqr(b)})\Jqr(b-u)+\frac{\Hqr(b)}{\Jqr(b)}\Jqr^\prime(b-u)-\Hqr^\prime(b-u)\\
=~&\Big(1-\frac{\delta}{\gamma+\delta}\frac{\Zqr(b)}{\Jqr(b)}\Big(\frac{\phiqr\Hqr(b)}{\Jqr(b)}+1\Big)\Big)\Jqr(b-u)+\frac{\Hqr(b)}{\Jqr(b)}\frac{\delta}{\gamma+\delta}\phiqr\Zqr(b-u)-\frac{\gamma}{\gamma+\delta}\Zq(b-u)\\
=~&\Jqr(b-u)-\frac{\delta}{\gamma+\delta}\Zqr(b-u)+\Big(\frac{\phiqr\Hqr(b)}{\Jqr(b)}+1\Big)\frac{\delta}{\gamma+\delta}\Zqr(b-u)\\
&~-\Big(\frac{\phiqr\Hqr(b)}{\Jqr(b)}+1\Big)\frac{\delta}{\gamma+\delta}\frac{\Zqr(b)}{\Jqr(b)}\Jqr(b-u)-\frac{\gamma}{\gamma+\delta}\Zq(b-u)\\
=~&\frac{\delta}{\gamma+\delta}\Big(\frac{\phiqr\Hqr(b)}{\Jqr(b)}+1\Big)\Big(\Zqr(b-u)-\frac{\Zqr(b)}{\Jqr(b)}\Jqr(b-u)\Big)\\
<~&0,
\end{align*}
provided that the mapping $b\mapsto\Zqr(b)/\Jqr(b)$ is an increasing function in $b$. \corrs{This is in turn true because we have from the definition of $\Jqr$ that, for $x\geq 0$,
$$\frac{\Zqr(x)}{\Jqr(x)}=\frac{\gamma+\delta}{\delta}\Bigg(\frac{1}{1+\frac{\gamma}{\delta}\frac{\Zq(x)}{\Zqr(x)}}\Bigg),$$ which is increasing due to \eqref{eqt.ZqrZq.increasing}.}
\end{proof}
\begin{proof}[Proof of Property 3]
We want to show $\lim_{b\rightarrow \infty}\corr{V_0(h;\pi_b)}=0$ for any $h\geq 0$. The case for $h=0$ is trivial. For $h>0$, Property 2 implies that $\lim_{b\rightarrow \infty}\corr{V_0(h;\pi_b)}$ exists. Using the fact that $V_0(h;\pi_b)\geq 0$, we can conclude that $\lim_{b\rightarrow \infty}V_0(h;\pi_b)$ exists. We shall show the limit is zero. Note that the smoothness condition for $b_l=0$ is equivalent to 
\begin{equation}\label{eqt.D3}
b_u-\frac{\Hqr(b_u)}{\Jqr(b_u)}=\kappa+\frac{\gamma}{\gamma+\delta}\frac{\mu}{\delta}.
\end{equation}
Since the left hand side is a strictly increasing function, for any $b>b^*$, there is a $\kappa$ corresponding to such $b$ such that the value function of the periodic $(b_u,\corr{0})$ strategy $\pi_{b,0}$ is smooth. As a result, we have $V_0(h;\pi_b)=V_\kappa(h;\corr{\pi_{b,0}})$. Furthermore, since $\lim_{b\rightarrow \infty}\frac{\Hqr(b)}{\Jqr(b)}$ exists, we can conclude that 
\[b-\kappa=\frac{\Hqr(b)}{\Jqr(b)}+\frac{\gamma}{\gamma+\delta}\frac{\mu}{\delta}<K\]
for large enough $K>0$. Now, we can write 
\[
V_0(h;\pi_b)=V_\kappa(h;\pi_{b,0})=\mathbb{E}_h[e^{-\delta T^+_b};T_{b_u}^+<\tau_0^-](b-\kappa)<K\mathbb{E}_h[e^{-\delta T^+_{b}};T_{b}^+<\tau_0^-],
\]
where $T^+_b=\min\{T_i:X(T_i)\geq b\}$. It should be clear that the expected value goes to $0$ when $b$ goes to infinity. Therefore, we can conclude $\lim_{b\rightarrow \infty}V_0(h;\pi_b)=0$ for any $h\geq 0$.
\end{proof}

Suppose when the transacion costs are $\kappa>0$, we have $b_l^*=0$, or equivalently $V_0^\prime(0;\pi_b)=V_\kappa^\prime(0;\pi_{b_u,0}^{\kappa,s})\leq 1$, where $b$ satisfies the smoothness condition (\ref{eqt.D3}). This means 
\begin{equation}
\lim_{h\rightarrow 0}\frac{V_0(h;\pi_b)}{h}\leq 1.
\end{equation}

Now, suppose further the fixed transaction costs increases to $\tilde{\kappa}>\kappa$. In view of Property 1, it means an increase in $\kappa$ on the r.h.s. of (\ref{eqt.D3}) must be compensated by an increase in $b_u$ on the r.h.s. of (\ref{eqt.D3}). Hence, we must choose a larger $b_u$ for $b_l=0$ to achieve smoothness. The new upper barrier is denoted as $\tilde{b}>b$. Since for a fixed $h>0$, $V_0(h;\pi_b)$ is decreasing in $b$ (Property 2), we have
\begin{equation}
V_\kappa^\prime(0;\pi_{b_u,0}^{\tilde{\kappa},s})=\lim_{h\rightarrow 0}\frac{V_0(h;\pi_{\tilde{b}})}{h}\leq\lim_{h\rightarrow 0}\frac{V_0(h;\pi_{b})}{h}\leq 1,
\end{equation}
which shows $b_l^*=0$ when the transaction costs are $\tilde{\kappa}$.

In summary, increasing $\kappa$ can never help to avoid $b_l^*=0$. As such, we shall choose
\begin{equation}
\kappa_0=\inf\{\kappa:\mbox{Liquidation-at-first-opportunity stategy is optimal}\},
\end{equation}
with the convention $\inf\emptyset=\infty$.
\end{document}